\documentclass[journal,twoside,web,romanappendices]{ieeecolor}
\usepackage{generic}
\usepackage{array}
\usepackage[caption=false,font=normalsize,labelfont=sf,textfont=sf]{subfig}
\usepackage{textcomp}
\usepackage{stfloats}
\usepackage{url}
\usepackage{verbatim}
\usepackage{amsmath} % assumes amsmath package installed
\usepackage{amsfonts,mathrsfs}
\usepackage{amssymb}
\usepackage{color}
\usepackage{graphicx}
\usepackage{epsfig}
\usepackage{enumerate}
\usepackage{algorithm,algorithmicx}
\usepackage{algpseudocode}
\usepackage[hidelinks]{hyperref}
\usepackage{marginnote}
\usepackage{empheq}
\usepackage{bm}
\usepackage{accents}
\usepackage{cite}
\usepackage{balance}
\newtheorem{definition}{Definition}{\it}{}
{\it}{}
\newtheorem{corollary}{Corollary}{\it}{}
\newtheorem{proposition}{Proposition}{\it}{}
\newtheorem{lemma}{Lemma}{\it}{}
\newtheorem{theorem}{Theorem}{\it}{}
\newtheorem{remark}{Remark}{\it}{}
\newtheorem{assumption}{Assumption}{\it}{}

\newcommand{\bs}{\boldsymbol}
\newcommand{\mc}{\mathcal}
\newcommand{\bb}{\mathbb}
\newcommand{\R}{\bb R}

\DeclareMathAlphabet{\mathbbmsl}{U}{bbm}{m}{sl}

\newcommand{\red}{\textcolor{red}}

\newcommand{\argmin}{\operatorname{argmin}}

\newcommand{\proj}{\mathrm{proj}}

\newcommand{\diag}{\operatorname{diag}}
\newcommand{\blkdiag}{\operatorname{blkdiag}}
\newcommand{\col}{\operatorname{col}}

\newcommand{\0}{\mathbf{0}}
\newcommand{\1}{\mathbf{1}}

\newcommand{\Rmnum}[1]{\expandafter\@slowromancap\romannumeral #1@}
\newcommand{\eod}{\ensuremath{\hfill\Box}}

 \def\QEDhereeqn{\eqno\let\eqno\relax\let\leqno\relax\let\veqno\relax\hbox{\QED}}
\def\QEDopenhereeqn{\eqno\let\eqno\relax\let\leqno\relax\let\veqno\relax\hbox{\QEDopen}}

\begin{document}

\title{A two-stage approach for a mixed-integer economic dispatch game in integrated electrical and gas  distribution systems}

\author{Wicak Ananduta and Sergio Grammatico
\thanks{W. Ananduta and S. Grammatico are with the Delft Center of Systems and Control (DCSC), TU Delft, the Netherlands. E-mail addresses: \texttt{\{w.ananduta, s.grammatico\}@tudelft.nl}. }
\thanks{This work was partially supported by the ERC under research project COSMOS (802348). }
}



\maketitle

\begin{abstract}
We formulate for the first time the economic dispatch problem in an integrated electrical and gas distribution system as a game equilibrium problem between distributed prosumers. Specifically, by approximating the non-linear gas-flow equations either with a mixed-integer second order cone or a piece-wise affine model and by assuming that electricity and gas prices depend linearly on the total consumption we obtain a potential mixed-integer game. To compute an approximate generalized Nash equilibrium, we propose an iterative two-stage method that exploits a problem convexification and the gas flow models. We quantify the quality of the computed solution and perform a numerical study to evaluate the performance of our method.
\end{abstract}


%%%%%%%%%%%%%%%%%%%%%%%%%%%%%%%%%%%%%%%%%%%%%%%%%%%%%%%%%%%%%%%%%%%%%%%%%%
\section{Introduction}

\IEEEPARstart{O}{ne} of the key features of future energy systems is the decentralization of power generation \cite{pepermans05distributed,mehigan18review}, where small-scale distributed generators (DGs) will have a major contribution in meeting energy demands. Furthermore, the intermittency of renewable power generation might necessitate the co-existence of non-renewable yet {controllable} DGs to provide enough supply and offer flexibility \cite[Sec. 3]{pepermans05distributed}. In this context, gas-fired generators, such as  combined heat and power (CHP) \cite{houwing10,zhang20}, can play a prominent role due to its efficiency and infrastructure availability. Consequently,  electrical and gas systems are {expected to be} more intertwined in the future {and} {in fact, this  integration has {received research} attention in the control {systems} community   \cite{zlotnik16,singh20,sundar21,hari22,roald20}.}

Several works, e.g. \cite{roald20,urbina07,correa14b,zhang16,wu20decentralized,he17,wen17,he18,liu19,li18privacy,qi19}, particularly  study {the} tertiary control problem {for} an integrated electrical and gas system (IEGS), {i.e., the problem of} {computing} optimal operating points {for the} generators. In these papers, the economic dispatch problem  is posed as an optimization program where the main objective is to minimize the operational cost  of the whole network, which includes electrical and gas production costs, {subject to physical {dynamics} and operational constraints.} 
Differently from  \cite{urbina07,correa14b,zhang16,wu20decentralized,he17,wen17,he18,liu19,li18privacy,qi19,roald20}, where a common objective is considered, when DGs are owned by different independent entities (prosumers), the operations of these DGs depend on {several individual} objectives. {In the latter case,} to compute optimal operating points of its DG, each prosumer must solve its own economic dispatch problem. However, these prosumers are coupled with each other as they share a common power (and possibly gas) distribution network. Their objective functions can also depend on the decisions of other prosumers, such as via an {energy} price function \cite{atzeni13,belgioioso21operationally}. {Therefore,} the economic dispatch problems of prosumers in an IEGS {is naturally} a generalized game \cite{facchinei10}. When each prosumer aims at finding a decision that optimizes its objective given the decisions of others, we obtain a generalized Nash equilibrium (GNE) problem, i.e., the problem of finding a GNE, a point where each player has no incentive to unilaterally deviate.

%\com{Transition to game setup is too fast. Better motivation is also required} Meanwhile, as there might potentially be multiple owners of the DGs in the network and each DG owner may be selfish, optimizing the social welfare, as discussed in the aforementioned literature, might not be a common objective. This fact renders the economic dispatch problem into a generalized game, i.e., a set of interdependent optimization problems, each of which has coupling cost function and constraints \cite{facchinei10}. A typical solution to a generalized game is a generalized Nash equilibrium (GNE) where each player has no incentive to unilaterally deviate. When the game is jointly convex, i.e., the cost function of each player is convex with respect to the player's strategy and the global feasible set is convex, efficient algorithmic solutions to solve a generalized Nash equilibrium problem are readily available in the literature, e.g., \cite{yi19,franci20,belgioioso20semi,bianchi22fast}. 

For a jointly convex game, i.e., {when} the cost function of each player is convex with respect to the player's strategy and the global feasible set is convex, efficient algorithmic solutions to solve a GNE problem are available in the literature, e.g., \cite{yi19,franci20,belgioioso20semi,bianchi22fast}.  However, the economic dispatch problem in an IEGS system is typically formulated as a mixed-integer  optimization program due to the approximation methods commonly used for non-linear gas-flow equations. For instance, \cite{he17,he18,wen17,liu19} consider mixed-integer second-order cone (MISOC) gas-flow models whereas  \cite{urbina07,correa14b,zhang16,wu20decentralized} use mixed-integer linear ones.  {On account of mixed-integer nature of the problem,} {the GNE seeking methods in \cite{yi19,franci20,belgioioso20semi,bianchi22fast} are  not applicable and in fact,} there are only a few works that propose GNE-seeking methods for mixed-integer generalized games, e.g., \cite{sagratella17,sagratella19,cenedese19,fabiani20}.

%\com{highlight novelty: game formulation} 
 In this paper, we formulate an economic dispatch problem in an integrated electrical and gas distribution system (IEGDS) as a mixed-integer generalized game (Section \ref{sec:IEGDSmodel}). Specifically, we consider prosumers, i.e., the entities that own active components, namely DGs and storage units, as the players of the game. Aside from consuming {electricity} and gas, each prosumer can produce and/or store {electrical} power as well as buy power and gas from the main grid, {where the prices depend linearly on the aggregate consumption.}  Furthermore, {our formulation can incorporate} {a} piece-wise affine (PWA) gas flow approximation, yielding a set of mixed-integer linear constraints, or {a classic} MISOC relaxation. {The} game-theoretic formulation {is the main conceptual novelty} of this {paper} compared to the existing literature in IEGDSs. {To our knowledge, there are only a few works, e.g. \cite{wang18,chen20}, that discuss  the dispatch of IEGDS systems as a (jointly convex) generalized game, {but} under a {substantially} different setup, i.e., a two-player game between the electrical and gas network operators, {subject to} a perfect-pricing assumption. }

Then, we propose a novel two-stage approach to compute a solution of the economic dispatch game, namely a(n) (approximate) mixed-integer GNE (MI-GNE) (Section \ref{sec:prop_approach}). In the first stage, we relax the problem into a jointly convex game and compute a GNE of the convexified game. {Next,} in the second stage we recover a mixed-integer solution, which has minimum gas-flow violation, by exploiting the {gas flow  models} and {by} solving a linear program. Furthermore, we can refine the computed solution by iterating these steps. {In these iterations, we introduce an auxiliary penalty function on the gas-flows to the convexified game and adjust its penalty weight. {Consequently,} we can provide a condition when our iterative algorithm  obtains an (approximate) MI-GNE {and measure the solution quality} (Theorem \ref{thm:character_sol}).} Differently from other existing MI-GNE seeking methods \cite{sagratella17,sagratella19,cenedese19,fabiani20}, our method allows for a parallel implementation {and does not solve a mixed-integer {optimization.}} {We also remark that existing distributed parallel mixed-integer optimization algorithms, e.g. \cite{falsone18,camisa22}, only deal with linear objective functions; therefore, they are unsuitable for our case.}  In Section \ref{sec:sim_res}, we show the performance of our algorithm via numerical simulations {of a benchmark} 33-bus-20-node distribution network. {We note that, in the preliminary work \cite{ananduta22}, we only consider the PWA gas flow model and implement the two stage approach without the refining iterations to compute an approximate solution to the economic dispatch problem of multi-area IEGSs \cite{wu20decentralized,he18,qi19}, which is an optimization problem with a common and separable cost function, instead of a {noncooperative} game.}%In {the} preliminary work \cite{ananduta22}, we show that {our} two-stage method without the refining iterations obtains solutions with better gas-flow approximation than those of {the} state-of-the-art method in \cite{he18} for the economic dispatch problem of multi-area IEGSs \cite{wu20decentralized,he18,qi19}, which is an optimization problem with a common and separable cost function, instead of a {noncooperative} game.
 
%\vspace{-5pt}
\subsection*{Notation} We denote by $\R$ and $\bb N$ the set of real numbers and that of natural numbers, respectively. We denote by  $\bs{0}$ ($\bs{1}$) a matrix/vector with all elements equal to $0$ ($1$).
%;
%to improve clarity, we may add the dimension of these matrices/vectors as subscript.
 The Kronecker product between the matrices $A$ and $B$ is denoted by $A \otimes B$. For a matrix $A \in \R^{n \times m}$, its transpose is $A^\top$. %, $[A]_{i,j}$ represents the element on the row $i$ and column $j$. %
For symmetric $ A \in \R^{n \times n}$, $A \succ 0$ ($\succcurlyeq 0$) stands for positive definite (semidefinite) matrix.
%Given $N$ vectors $x_1 \in \bb R^{n_1}, x_2 \in \bb R^{n_2},\ldots, x_N \in \R^{n_N}$, $x := \col\left(x_1,\ldots,x_N\right) = [ x_1^\top, \ldots , x_N^\top ]^\top$.
%{For any $x \in \bb R^n$, $\|x\|_A^2 = x^\top A x,$ with square symmetric matrix $A \succ 0$}. 
{The operator $\col(\cdot)$ stacks its arguments into a column vector whereas  $\diag(\cdot)$ ($\blkdiag(\cdot)$) creates a (block) diagonal matrix with its arguments as the (block) diagonal elements.} The sign operator is denoted by $\operatorname{sgn}(\cdot)$, i.e., 

\vspace{-5pt}
{ $$ \operatorname{sgn}(a) = \begin{cases}
	1 \quad &\text{if } a > 0, \\
	0 \quad &\text{if } a = 0, \\
	-1 \quad &\text{if } a < 0.
\end{cases}$$}
%\com{monotonicity, convexity, potential game, exact potential function}
\section{Economic dispatch game}

\label{sec:IEGDSmodel}
In this section, we formulate  {the} economic dispatch game of a set of $N$ prosumers (agents), denoted by $\mc I := \{1,2,\dots,N\}$, in an IEGDS. Each  prosumer seeks an economically efficient decision {(optimal references/set points)} to meet its electrical and gas demands over a certain time horizon, denoted by $H$; let us denote the set of time indices by $\mc H:=\{1,\dots, H\}$. First, we provide the model of the system, which consists of two parts, the electrical and gas networks.

\subsection{Electrical network}
 {To} meet the electrical demands, denoted by $d_i^{\mathrm e} \in \bb R^H_{\geq 0}$ for all $i \in \mc I$, {the} prosumers may have a dispatchable DG, which can be either gas-fueled or non-gas-fueled.  We denote the set of agents that have a \underline{g}as-fueled \underline{u}nit by $\mc I^{\mathrm{gu}} \subset \mc I$ whereas those that have a \underline{n}on-\underline{g}as-fueled \underline{u}nit by $\mc I^{\mathrm{ngu}} \subset \mc I$. We note that $\mc I^{\mathrm{dg}}:=\mc I^{\mathrm{gu}} \cup \mc I^{\mathrm{ngu}}  \subseteq \mc I$. For each $i \in \mc I^{\mathrm{dg}}$, let us denote the power produced by its  {generation} unit by $p_i^{\mathrm{dg}} \in \bb R_{\geq 0}^H$, constrained by
\begin{equation}
	\begin{aligned}
		\1_H\,\underline{p}_{i}^{\mathrm{dg}} \leq p_{i}^{\mathrm{dg}} &\leq \1_H \, \overline{p}_{i}^{\mathrm{dg}}, 
	\end{aligned}
	\label{eq:p_pu_bound}	
\end{equation}
where $\underline{p}_{i}^{\mathrm{dg}} < \overline{p}_{i}^{\mathrm{dg}}$ denote the minimum and maximum power  {production.} Specifically for the prosumers with non-gas-fueled  {DGs,} we consider a quadratic cost of producing power, i.e., 
\begin{equation}
	f_{i}^{\mathrm{ngu}}(p_i^{\mathrm{dg}}) = 
	\begin{cases}
		q_i^{\mathrm{ngu}}\|p_i^{\mathrm{dg}}\|^2+ l_i^{\mathrm{ngu}} \1^\top p_i^{\mathrm{dg}}, \  &\text{if } i \in \mc I^{\mathrm{ngu}}, \\
		0, \quad & \text{otherwise,}
	\end{cases}
	\label{eq:f_ngu}
\end{equation}
where $q_i^{\mathrm{ngu}}>0$ and $l_i^{\mathrm{ngu}}$ are constants. On the other hand, for the prosumers with gas-fueled  {DGs,} we assume a linear relationship between the consumed gas, $d_i^{\mathrm{gu}}$, and the  {produced power,} as in  {\cite[Eq. (24)]{liu19},} i.e., 
\begin{equation}
	d_i^{\mathrm{gu}} =
	\begin{cases}
		\eta_i^{\mathrm{gu}} p_i^{\mathrm{dg}}, \quad &\text{if } i \in \mc I^{\mathrm{gu}}, \\
		0, \quad &\text{otherwise},
	\end{cases}
	\label{eq:conv_dgu_ppu}	
\end{equation}
where $\eta_i^{\mathrm{gu}}>0$ denotes the conversion factor. 

%\paragraph*{electrical storage units} 
Each prosumer might also own a  {controllable} storage unit, whose cost function, which corresponds to reducing its degradation, is denoted by $f_{i}^{\mathrm{st}}:\mathbb{R}^{2H} \to \mathbb{R}$:
\begin{equation}
	f_{i}^{\mathrm{st}}(p_{i}^{\mathrm{ch}},p_{i}^{\mathrm{dh}})= %\|p_{i}^{\mathrm{ch}}\|_{Q_i^{\mathrm{st}}}^2 + \|p_{i}^{\mathrm{dh}}\|_{Q_i^{\mathrm{st}}}^2, \label{eq:f_st}
	(p_{i}^{\mathrm{ch}})^\top Q_i^{\mathrm{st}}p_{i}^{\mathrm{ch}} + (p_{i}^{\mathrm{dh}})^\top Q_i^{\mathrm{st}}p_{i}^{\mathrm{dh}} , \label{eq:f_st}
\end{equation}
where $Q_i^{\mathrm{st}}\succcurlyeq0$. The variables $p_{i}^{\mathrm{ch}} = \col((p_{i,h}^{\mathrm{ch}})_{h\in \mc H})$ and $p_{i}^{\mathrm{dh}} = \col((p_{i,h}^{\mathrm{dh}})_{h\in \mc H})$ denote the charging and discharging powers, which are constrained by the battery dynamics {and operational limits}  {\cite[Eqs. (1)--(3)]{zhang20}:}
\begin{equation}
	\begin{aligned}
		\left.
		\begin{array}{c}
			x_{i,h+1} =\eta_i^{\mathrm{st}} x_{i,h} +\frac{T_{\mathrm{s}}}{e^{\mathrm{cap}}_{i}}(\eta_i^{\mathrm{ch}} p_{i,h}^{\mathrm{ch}} -(\frac{1}{\eta_i^{\mathrm{dh}}})p_{i,h}^{\mathrm{dh}}), \\[.2em]
			\underline{x}_i \leq x_{i,h+1} \leq \overline{x}_i, \\[.2em]
		\end{array}
		\hspace{-5pt} \right\} & \begin{matrix}
			\forall i \in \mc I^{\mathrm{st}}, \\ \forall h \in \mc H,
		\end{matrix} \\
	{p}_{i}^{\mathrm{ch}} \in [0,\overline{{p}}_{i}^{\mathrm{ch}}], \ \ {p}_{i}^{\mathrm{dh}} \in [0,\overline{{p}}_{i}^{\mathrm{dh}}], \qquad \qquad & \forall i \in \mc I^{\mathrm{st}}, \\
		p_{i}^{\mathrm{ch}}= 0, \ \ p_{i}^{\mathrm{dh}}= 0,\qquad \qquad & \forall i \in {\mc I \backslash \mc I^{\mathrm{st}},}
	\end{aligned}
	\label{eq:storage}
\end{equation}
where $x_{i,h} $ denotes the state of charge (SoC) of the storage unit at time $h \in \mc H$, $\eta_i^{\mathrm{st}}, \eta_i^{\mathrm{ch}}, \eta_i^{\mathrm{dh}} \in (0,1]$ denote  the leakage coefficient of the storage, charging, and discharging efficiencies, respectively, while %$b_i~=~-\frac{T_{\mathrm{s}}}{e^{\mathrm{cap}}_{i}}$, with 
$T_{\mathrm{s}}$ and $e^{\mathrm{cap}}_{i}$ denote the sampling time and the maximum capacity of the storage, respectively. Moreover, $\underline{x}_i,\overline{x}_i  \in [0,1]$  denote the minimum and the maximum SoC of the storage unit of prosumer $i$, respectively, whereas $\overline p^{\mathrm{ch}}_i \geq 0$ and $\overline p^{\mathrm{dh}}_i \geq 0$ denote the maximum charging and discharging power of the storage unit. Finally, we denote by  $\mathcal{I}^{\mathrm{st}}\subseteq\mathcal{I}$ 	the set of prosumers that own a storage unit.

These prosumers may also buy electrical power from the main grid, and we denote this decision by $p_i^{\mathrm{eg}} := \col((p_{i,h}^{\mathrm{eg}})_{h \in \mc H})\in \bb R^H_{\geq 0}$, where $p_{i,h}^{\mathrm{eg}}$ denotes the decision at time step $h$. %As in \cite{atzeni2012demand}, 
We {consider a typical assumption in demand-side management,} {namely that} the electricity price depends on the total consumption of the network of prosumers,  {which} is usually defined as a quadratic function  {\cite[Eq. (12)]{atzeni13},} i.e., 
\begin{equation*}
	c_h^{\mathrm{e}}(\sigma^{\mathrm{e}}_h)= q_h^{\mathrm{e}}(\sigma^{\mathrm{e}}_h)^2 + l_h^{\mathrm{e}}\sigma^{\mathrm{e}}_h,
\end{equation*}
where $\sigma^{\mathrm{e}}_h$ denotes the {aggregate load on the main electrical grid}, i.e., 
%\begin{equation*}% \label{eq:Aggr}
$	\sigma^{\mathrm{e}}_h = \sum_{i\in\mathcal{I}} p_{i,h}^{\mathrm{eg}},$
%\end{equation*}
and $q_h^{\mathrm{e}}, l_h^{\mathrm{e}} \geq 0$ are constants.
Therefore, by denoting $\sigma^{\mathrm{e}} = \col(\{\sigma^{\mathrm{e}}_h\})$, the objective function of agent $i$ associated to the trading with the main grid, denoted by $f_{i}^{\mathrm{eg}}$, is defined as 
\begin{equation}
	\begin{aligned} \textstyle
		f_{i}^{\mathrm{e}}\left( p_{i}^{\mathrm{eg}},\sigma^{\mathrm{e}} \right) &= \sum_{h \in \mc H} c_h^{\mathrm{e}}(\sigma^{\mathrm{e}}_h) \frac{p_{i,h}^{\mathrm{eg}}}{\sum_{j\in\mathcal{I}}p_{j,h}^{\mathrm{eg}}}\\
		&= \sum_{h \in \mc H} q_h^{\mathrm{e}}\cdot\Big(\sum_{j\in\mathcal{I}}p_{j,h}^{\mathrm{eg}}\Big)\cdot p_{i,h}^{\mathrm{eg}} + l_h^{\mathrm{e}}. \label{eq:f_mg}
	\end{aligned}
\end{equation}
Moreover, we impose that the aggregate power traded with the main grid is bounded as follows:
\begin{align}
	\1_H \,\underline{\sigma}^{\mathrm{e}} \leq \sum_{i\in\mathcal{I}} p_{i}^{\mathrm{eg}}\leq \1_H \, \overline{\sigma}^{\mathrm{e}}, \label{eq:p_mg_bound}
\end{align}
where $\overline{\sigma}^{\mathrm{e}}>\underline{\sigma}^{\mathrm{e}} \geq 0$ denote the upper and lower bounds. Note that the lower bound might be required to be positive in order to ensure a continuous operation of the main generators that supply the main grid.

%\com{I think we should include the power flow equations instead of having peer-to-peer trading. Adding power flow equations means we will need another layer/graph. However, having three different graphs seem too complicated. Therefore, I prefer to remove the trading feature and graph, as our emphasis in this work is in the interconnection between gas and electrical network.}

Next, we describe the physical constraints of the electrical network. {For ease of exposition,} we assume that each agent is associated with a bus (node) in an electrical distribution network, which can be represented by an  undirected graph $\mc G^{\mathrm{e}}:=(\mc I, \mc E^{\mathrm e})$, where $\mc E^{\mathrm e}$ denotes the set of power lines. Therefore, we denote by $\mc N_i^{\mathrm e}$ the set of neighbor nodes of node $i$, i.e., $\mc N_i^{\mathrm e} :=\{j \mid (i,j) \in \mc E^{\mathrm e}\}$. Let us then denote by $\theta_i \in \bb R^H$ and $v_i\in \bb R^H$ the voltage angle and magnitude of bus $i \in \mc I$ while by $p_{(i,j)}^{\ell},q_{(i,j)}^{\ell} \in \bb R^H$ the active and reactive powers of line $(i,j) \in \mc E^{\mathrm e}$. 

{The voltage phase angle ($\theta_i$) and magnitude ($v_i$) for each node $i \in \mc I$ are bounded by}
\begin{subequations}
	\label{eq:line}
	\begin{align}
	%	(p_{(i,j),h}^{\ell})^2 + (q_{(i,j),h}^{\ell})^2 &\leq \overline{s}_{(i,j)}^2, \forall j\in\mc N_i^{\mathrm e}, \ \forall h\in \mc H, \label{eq:line_cap}\\
		\underline{\theta}_i\1 \leq \theta_i &\leq \overline{\theta}_i\1,  \label{eq:theta_b}\\
		\underline{v}_i\1 \leq v_i &\leq \overline v_i\1, \label{eq:v_b}
	\end{align}
\label{eq:rel_cons}%
\end{subequations}
where %\eqref{eq:line_cap} represents the line capacity constraint at each line, with maximum capacity of line $(i,j) \in \mc E^{\mathrm e}$ denoted by $\overline{s}_{(i,j)}$, and
 %\eqref{eq:theta_b}--\eqref{eq:v_b} represent the bounds of the voltage phase angles and magnitudes, respectively, with 
 $\underline{\theta}_i \leq \overline{\theta}_i$ denote the minimum and maximum phase angles and $\underline{v}_i\leq \overline v_i$ denote the minimum and maximum voltages. %Note that, when linearizing the power flow equations, we take one of the busses as reference bus. 
 Without loss of generality, we suppose the reference is bus $1$ and assume $\underline{\theta}_1 = \overline{\theta}_1=0$. 

The power balance equation, which ensures equal production and consumption,  at each bus can be written as:
\begin{subequations}
\begin{align}
	d_i^{\mathrm{e}} &= p_i^{\mathrm{dg}} + p_i^{\mathrm{eg}} + p_i^{\mathrm{dh}} - p_i^{\mathrm{ch}} ,  & \forall i \in \mc I, \label{eq:pow_bal} \\
	p_i^{\mathrm{eg}} &= p_i^{\mathrm{et}}-\sum_{j \in \mc N_i} p_{(i,j)}^{\ell},  &\forall i \in \mc I, \label{eq:pow_bal2}
\end{align}
\label{eq:pow_bals}%
\end{subequations}
where $p_i^{\mathrm{et}}$ denotes the injected power from the electrical transmission grid if node $i$ is connected to it and $p_{(i,j)}^{\ell}$, for each $j \in \mc N_i$, denote the real power line between node $i$ and its neighbor $j$. 
Furthermore, for simplicity, we use a linear approximation of the real power-flow equations \cite[Eq. (2)]{yang19}, 
%\begin{subequations}
	%\small
%	\label{eq:pf_pq}
	\begin{align}
		&p_{(i,j)}^{\ell} = B_{(i,j)}\left({\theta_i - \theta_j} \right) - G_{(i,j)}\left(v_i - v_j \right), \, \, \forall j \in \mc N_i^{\mathrm e}, \label{eq:pf_all}%\\
		%&q_{(i,j)}^{\ell} = G_{(i,j)}\left({\theta_i - \theta_j} \right) + B_{(i,j)}\left(v_i - v_j \right), \, \,  \forall j \in \mc N_i^{\mathrm e}, \label{eq:pf_q}
	\end{align}
%	\label{eq:pf_all}%
%\end{subequations}
for each $i \in \mc I$, where $B_{(i,j)}$ and  $G_{(i,j)}$ denote the absolute values of the susceptance and conductance of line $(i,j)$, respectively. %Note that by \eqref{eq:pf_all}, for each pair $(i,j) \in \mc E^{\mathrm e}$, it holds that $p_{(i,j)}^{\ell} = {-p}_{(j,i)}^{\ell}$. % and $q_{(i,j)}^{\ell} = {-q}_{(j,i)}^{\ell}$. 
Finally, let us now collect {all the decision variables associated with the electrical network by} $x_i := \col(p_i^{\mathrm{dg}},p_i^{\mathrm{ch}}, p_i^{\mathrm{dh}}, p_i^{\mathrm{eg}},d_i^{\mathrm{gu}}, \theta_i, v_i, p_i^{\mathrm{et}},  \{p_{(i,j)}^{\mathrm{\ell}}\}_{j \in \mc N_i^{{\mathrm e}}}) \in \bb R^{n_{x_i}}$, with $n_{x_i}=H(8+\mc N_i^{\mathrm e})$,  for each $i \in \mc N$.

\subsection{Gas network}
%\paragraph*{Define graph of gas network}	
 Beside consuming $d_i^{\mathrm{gu}}\in \bb R_{\geq 0}^H$ for its gas-fired generator, we suppose that prosumer $i$ has  an {undispatchable} gas demand, denoted by $d_i^{\mathrm{g}}\in \bb R_{\geq 0}^H$.  The total gas demand of each prosumer is satisfied by buying gas from a source, which can either be a gas transmission network or a gas well. These {prosumers} are connected in a gas distribution network, represented by an undirected graph denoted by $\mc G^{\mathrm{g}} = (\mc I, \mc E^{\mathrm{g}})$, where  we assume that  each agent is a different node in $\mc G^{\mathrm{g}}$ and denote by $\mc E^{\mathrm{g}}$  the set of pipelines (links), where both $(i,j),(j,i) \in \mc E^{\mathrm g}$ represent the pipeline between nodes $i$ and $j$.  If node $j$ is connected to node $i$, then node $j$ belongs to the set of neighbors of node $i$ in the gas network, denoted by $\mc N_i^{\mathrm{g}} := \{ j \in \mc N^{\mathrm{g}} \mid (i,j) \in  \mc E^{\mathrm{g}}\}$. Therefore, the gas-balance equation of node $i \in {\mc I}$ can be written as: %\red{[add new var for importing gas from transmission grid. Reformulate coupling constraints etc.]}
\begin{align}
	g_i^{\mathrm{s}} - d_i^{\mathrm{g}} - d_i^{\mathrm{gu}} &= \sum_{j \in \mc N_i^{\mathrm{g}}} \phi_{(i,j)},  \label{eq:gbal}%\\
	%	\sum_{j \in \mc N_i^{\mathrm{g}}} \phi_{(i,j)} &= d^{\mathrm{tg}}, \quad i=N+1,\label{eq:gbal1}
\end{align}
where $g_i^{\mathrm{s}}\in \bb R_{\geq 0}^H$ denotes the imported gas from a source, if agent $i$ is connected to it. We denote the set of nodes connected to the gas source by $\mc I^{\mathrm{gs}}$. {Moreover,}	$\phi_{(i,j)}:=\col((\phi_{(i,j),h})_{h\in\mc H}) \in \bb R^H$ denotes the flow between nodes $i$ and $j$ from the perspective of agent $i$, i.e., $\phi_{(i,j),h} >0$ implies the gas flows from node $i$ to node $j$.  %We note that by \eqref{eq:gbal}--\eqref{eq:gbal1}, it holds that
%	\begin{equation}
	%	\sum_{i\in\mathcal{N}^{\mathrm g}} d_i^{\mathrm{g}} + d_i^{\mathrm{gu}} = \sum_{i\in\mathcal{N}^{\mathrm g}} \sum_{j \in \mc N_i^{\mathrm{g}}} \phi_{(i,j)} = d^{\mathrm{tg}}.
	%	\label{eq:Aggr_gas}
	%	\end{equation}
We formulate the cost of buying gas similarly to that of importing power from the main grid as these prosumers must pay the gas with a common price that may vary. Specifically, with the per-unit cost that depends on the total gas consumption, we have the following cost functions  {($\forall i \in \mc I$):}
\begin{equation}
	f_i^{\mathrm g}(d_i^{\mathrm{gu}},\sigma^{\mathrm{g}}) = \sum_{h \in \mc H} q_h^{\mathrm{g}} \cdot \sigma_h^{\mathrm{g}} \cdot (d_{i,h}^{\mathrm{gu}}+d_{i,h}^{\mathrm{g}}) + l_h^{\mathrm{g}}, 
\end{equation}
where $q_h^{\mathrm{g}}>0$ and $ l_h^{\mathrm{g}} \in \R$ are the cost parameters whereas $\sigma^{\mathrm{g}} = \col(\{\sigma_h^{\mathrm{g}}\}_{h \in \mc H})$ with 
%\begin{equation*}
$	\sigma_h^{\mathrm{g}} = \sum_{i \in \mc I} (d_{i,h}^{\mathrm{ gu}}+ d_{i,h}^{\mathrm{g}}),$  %\label{eq:agg_gas}
%\end{equation*}
which denotes the aggregated gas demand. 
In addition, the following constraints on the gas network are typically considered:
%	\paragraph*{Constraints in gas flow networks}
%The constraints that we consider are:
\begin{itemize}
	\item Weymouth gas-flow equation for two neighboring nodes:
	\begin{equation}
		\phi_{(i,j),h} = \operatorname{sgn}(\psi_{i,h}-\psi_{j,h}) c_{(i,j)}^{\mathrm f}\sqrt{|\psi_{i,h}-\psi_{j,h}|},  \label{eq:gf_eq}
	\end{equation}
	for all $h \in \mc H$, $ j \in \mc N_i^{\mathrm{g}}$, and $i \in \mc I^{\mathrm{g}}$, where $\psi_{i,h} \in \bb R_{\geq 0}$ is the squared pressure at node $i$, and $c_{(i,j)}^{\mathrm f}>0$ is some constant. We define $\psi_i = \col((\psi_{i,h})_{h\in\mc H})$. {By assuming a sufficiently large sampling time, we consider static gas flow equations as in, e.g. \cite{he17,wen17, he18, liu19}, instead of dynamic ones such as \cite{roald20,hari22}.}
	
	%We note that $\phi_{(i,j)} $ is assigned to node $i$ whereas $\phi_{(j,i)} $ is assigned to node $j$. 
	\item Bounds on  the gas flow $\phi_{(i,j)}$ and the pressure $\psi_i$:
	\begin{align}
		%\underline s_i^{\mathrm{g}} & \leq s_i^{\mathrm{g}} \leq \overline s_i^{\mathrm{g}} ,\quad  \forall i \in \mc N^{\mathrm{g}} \label{eq:sply_lim}\\
		\hspace{-15pt} - \1_H \, \overline \phi_{(i,j)}  \leq \phi_{(i,j)} \leq \1_H \,\overline \phi_{(i,j)}, \ \ &\forall j\in\mc N_i^{\mathrm{g}}, i \in \mc I, \label{eq:flow_lim}\\
		\1_H \,\underline \psi_i  \leq \psi_i \leq \1_H \,\overline \psi_i,\ \  &\forall i \in \mc I, \label{eq:pres_lim} \\
		g_i^{\mathrm{s}}= 0, \ \ &\forall i \in {\mc I \backslash \mc I^{\mathrm{gs}},} \label{eq:gt}
	\end{align}
	where $ \overline \phi_{(i,j)}$ denotes the maximum flow of the link $(i,j) \in \mc E^{\mathrm g}$ whereas  {$\underline \psi_i$ and  $\overline \psi_i$} denote the minimum and maximum {(squared)} gas pressure of node $i$,  {respectively.} The constraint  {in} \eqref{eq:gt} ensures that gas only flows through the {nodes} that are connected to  {a gas source.} 
	\item Bounds on the total gas consumption of the network $\mc G^{\mathrm g}$:
	\begin{equation}
		\1_H\underline{\sigma}^{\mathrm{g}} \leq \sum_{i\in\mathcal{N}^{\mathrm g}} (d_i^{\mathrm{g}}  + d_i^{\mathrm{dg}}) \leq  \1_H\overline\sigma^{\mathrm{g}},
		\label{eq:dtg_bound}
	\end{equation}
	where  {$\underline{\sigma}^{\mathrm{g}}$ and $\overline \sigma^{\mathrm{g}}$} denote the minimum and maximum total gas consumption of the distribution network $\mc G^{\mathrm g}$,  {respectively.}
\end{itemize}

\subsection{Approximation models of gas-flow equations}
\label{sec:gf_approx}
The gas-flow equations in \eqref{eq:gf_eq} are nonlinear and in fact introduce non-convexity to the decision problem. In this work, {we consider two models that are commonly used in the literature, namely the  MISOC relaxation and the PWA approximation. In the former, the gas flow equation is reformulated and relaxed into inequality constraints,  whereas in the latter it is approximated by a PWA function. Both models require the introduction of auxiliary continuous and binary variables, collected in the vectors $y_i$ and $z_i$, for each $i\in \mc I$, respectively. For ease of presentation, we represent the two models as a set of equality and inequality constraints:} 
\begin{align}
	h_{i}^{\mathrm{cpl}}(y_i,\{y_j\}_{j \in \mc N_i^{\mathrm g}}) = 0, \ \ &\forall i \in \mc I,\label{eq:hf_cons2}\\
	h_{i}^{\mathrm{loc}}(y_i,z_i) = 0,\ \ &\forall i \in \mc I,\label{eq:hf_cons1}\\
	g_{i}^{\mathrm{cpl}}(y_{i},z_{i}, \{y_j\}_{j \in \mc N_i^{\mathrm g}}) \leq 0,\ \ &\forall i \in \mc I, \label{eq:gf_cons2}\\
	g_{i}^{\mathrm{loc}}(y_{i},z_{i}) \leq 0,\ \ &\forall i \in \mc I, \label{eq:gf_cons1}
\end{align}
where  \eqref{eq:hf_cons2} and \eqref{eq:gf_cons2} are coupling constraints since $h_{i}^{\mathrm{cpl}}$ and $g_{i}^{\mathrm{cpl}} $ depend on the decision variables of the neighbors in $\mc N_i^{\mathrm g}$ while \eqref{eq:hf_cons1} and \eqref{eq:gf_cons1} are local constraints. 

We now briefly explain the MISOC and PWA models and introduce their auxiliary variables.
\paragraph{MISOC model}
We can obtain the MISOC model by relaxing the gas-flow constraints in \eqref{eq:gf_eq} into inequality constraints, introducing a binary variable to indicate each flow direction, and using {the} McCormick envelope to substitute the product of two decision variables with an auxiliary variable, denoted by $\nu_{(i,j)} \in \bb R^H$, for each $j \in \mc N_i^{\mathrm g}$ and $i \in \mc I$. We provide the detailed derivation in Appendix \ref{ap:soc_gf}. For this model,  $y_{i}:= \col(\psi_i, g_i^{\mathrm{s}}, \{\phi_{(i,j)}, \nu_{(i,j)} \}_{j \in \mc N_i^{\mathrm{g}}} )  \in \bb R^{n_{y_i}}$, with $n_{y_i}=H(2+2|\mc N_i^{\mathrm{g}}|)$, concatenates the physical variables of the gas network, i.e., $\psi_i$, $g_i^{\mathrm s}$, and $\phi_{(i,j)}$, for all $j \in \mc N_i^{\mathrm g}$, with the auxiliary variables $\nu_{(i,j)}$,  whereas  $z_i := \col(\{\delta_{(i,j)} \}_{j \in \mc N_i^{\mathrm{g}}} ) \in \{0,1\}^{n_{z_i}}$, with $n_{z_i}=H|\mc N_i^{\mathrm{g}}|$,  collects the binary decision vectors that indicate the flow directions. The coupling constraints for this model are affine. Furthermore, this model includes a set of convex second order cone (SOC) local constraints and does not have any local equality constraints. 
We note that if the gas flow directions (binary variables) are known, then the model {becomes} convex. Furthermore, when the SOC local constraints of the MISOC model are tight, the original gas flow constraints in \eqref{eq:gf_eq} are satisfied. However, we cannot guarantee the tightness of the SOC constraints in general, although one can use a penalty-based method \cite{he18} or sequential cone programming method \cite{wang18,he18} to induce tightness. %On the other hand, unlike the MISOC model, 

\paragraph{PWA model}
We obtain the PWA model by approximating the mapping $\phi_{(i,j),h} \mapsto (1/{c_{(i,j)}^{\mathrm f}})^2{\phi_{(i,j),h}^2}$, for each $(i,j) \in \mc E^{\mathrm g}$, with $r$ pieces of affine functions. Then we can use this approximation in \eqref{eq:gf_eq}. Furthermore, by utilizing the mixed-logical constraint reformulation \cite{bemporad99}, we obtain an approximated model of the gas-flow constraints as a set of mixed-integer linear constraints, as detailed in Appendix \ref{ap:pwa_gf}. %that involve the main decision variable $x_i$, some auxiliary variables $y_i$ and $z_i$, as well as the neighbor decision variable $x_j$, for all $j \in \mc N_i^{\mathrm g}$, i.e., for each $i \in \mc I$,
%for each $i \in \mc I$,
 In this model, for each $i \in \mc I$, we introduce some auxiliary continuous variables {$\nu_{(i,j)}^{\psi} \in \bb R^H$, for all $j \in \mc N_i^{\mathrm g}$, and  $\nu_{(i,j)}^m \in \bb R^H$, for $m=1,\dots,r$ and all $j \in \mc N_i^{\mathrm g}$, and thus, define $y_{i}:= \col(\psi_i, g_i^{\mathrm{s}}, \{\phi_{(i,j)}, \nu_{(i,j)}^{\psi}, \{\nu_{(i,j)}^m \}_{m=1}^r \}_{j \in \mc N_i^{\mathrm{g}}} )  \in \bb R^{n_{y_i}}$, with $n_{y_i}=H(2+(2+r)|\mc N_i^{\mathrm{g}}|)$.} Moreover, the auxiliary variable $z_i := \col(\{\delta_{(i,j)}, \{\alpha_{(i,j)}^m, \beta_{(i,j)}^m, \gamma_{(i,j)}^m\}_{r=1}^m \}_{j \in \mc N_i^{\mathrm{g}}} ) \in \{0,1\}^{n_{z_i}}$ collects the binary decision vectors, with $n_{z_i}=H(1+3r)|\mc N_i^{\mathrm{g}}|$. The variable $\delta_{(i,j)}$ is the indicator of gas flow direction in $(i,j) \in \mc E^{\mathrm g}$ while the remaining variables define the active region of the PWA approximation function. We note that, in this model, all the constraints are affine, unlike {in} the MISOC model. On the other hand, the latter requires significantly less number of auxiliary variables than the PWA model.  In addition, the approximation accuracy of the PWA model can be controlled a priori by the model parameter $r$ (see \cite[Section IV]{ananduta22} for a numerical study).

\color{black}

\subsection{Generalized potential game {formulation}}

%Based on the model of the IEGDS explained in Section \ref{sec:IEGDSmodel}, 
We can now formulate the economic dispatch problem of an IEGDS as a generalized game. {The formulation is applicable for both gas flow models explained in Section \ref{sec:gf_approx}. To that end,} let us denote the decision variable of agent $i$ by $u_i :=(x_i,y_i,z_i)$ and {the collection of decision variables of all agents by $\bm u=(\bm x,\bm y, \bm z)$, where $\bm x = \col((x_i)_{i \in \mc I})$ ($\bm y$ and $\bm z$ are defined) similarly.} We can formulate the interdependent optimization problems of the economic dispatch as follows: 
\begin{subequations}
	\begin{empheq}[left=\forall i \in \mc I \empheqlbrace]{align}
		\label{eq:cost_gen}	
		& 
		\underset{u_i :=({x}_i, y_i, z_i)}{\min}  \ \ \  J_{i}\left(
		\bm x
		\right) \\[.2em]
		\label{eq:const_gen1}
		&\quad \text{s.t. } \; \quad u_i \in  \mc L_i, z_i \in \{0,1\}^{n_{z_i}},\\
		&\qquad \quad \; \; \text{\eqref{eq:p_mg_bound}, \eqref{eq:pf_all}, \eqref{eq:dtg_bound},  \eqref{eq:hf_cons2}, and \eqref{eq:gf_cons2}.} \label{eq:coup_const}% \\
		%
		%	& \qquad \qquad
		%	u_{i} \in \mc C_i(u_{-i})	
	\end{empheq}
	\label{eq:ED_game1}%
\end{subequations} 
The cost function of agent $i$ in \eqref{eq:cost_gen}, $J_i$, is composed by the local function $ f_i^{\mathrm{loc}}$ and  {the} coupling function $f_i^{\mathrm{cpl}}$, i.e.:
\begin{align}
	J_i(\bm x) &:= f_i^{\mathrm{loc}}(x_i) + f_i^{\mathrm{cpl}}(\bm x), \label{eq:Ji} \\
	f_i^{\mathrm{loc}}(x_i) &:= f_{i}^{\mathrm{ngu}}(p_i^{\mathrm{dg}}) + f_{i}^{\mathrm{st}}(p_{i}^{\mathrm{ch}},p_{i}^{\mathrm{dh}}) 
	, %+ \red{f_{i}^{\mathrm{tr}} ( \{ p_{(i,j)}^{\mathrm{tr}} \}_{j \in \mc N_i} ) }, 
	\label{eq:f_loc} \\
	f_i^{\mathrm{cpl}}(\bm x) &:=	f_{i}^{\mathrm{e}}( p_{i}^{\mathrm{eg}},\sigma^{\mathrm{e}}(\bm x) ) + f_i^{\mathrm g}(d_i^{\mathrm{gu}},\sigma^{\mathrm{g}}(\bm x)),  \label{eq:f_cpl}
\end{align}	 
where $\sigma^{\mathrm{e}}(\bm x)$ and $\sigma^{\mathrm{g}}(\bm x)$ depend on the decision variables of all agents. %as shown in \eqref{eq:Aggr} and \eqref{eq:agg_gas}. 
The local set $\mc L_i \subset \bb R^{n_i}$ in \eqref{eq:const_gen1}, with $n_i =n_{x_i} + n_{y_i} + n_{z_i} $, is defined by 
\begin{multline}
	\mc L_i := \left\{ u_i \in \bb R^{n_i} \mid \text{\eqref{eq:p_pu_bound}, \eqref{eq:conv_dgu_ppu}, \eqref{eq:storage}, \eqref{eq:rel_cons}, \eqref{eq:pow_bals}, \eqref{eq:gbal},}\right.\\
		\left.\text{ \eqref{eq:flow_lim}, \eqref{eq:pres_lim}, \eqref{eq:gt}, \eqref{eq:hf_cons1}, and \eqref{eq:gf_cons1} hold} \right\}, \label{eq:Li}
\end{multline}
whereas the equalities and inequalities stated in \eqref{eq:coup_const} define the coupling constraints of the game. {By the definitions of the constraints, including any of the gas flow approximation models, $\mc L_i$ is convex.} %We define the coupling set by 
%\begin{equation*}
%	\mc C = \{\bm u \in \bb R^n \mid \text{\eqref{eq:p_mg_bound}, \red{\eqref{eq:rep_const},} \eqref{eq:dtg_bound},  \eqref{eq:hf_cons2}, and \eqref{eq:gf_cons2} hold} \}.
%\end{equation*} 
However, due to the binary variables, $z_i$, for all $i \in \mc I$, the game in \eqref{eq:ED_game1} is mixed-integer. Moreover, let us consider the following technical assumption.% \cite{bibid}.
\begin{assumption}
	\label{as:nonempty_globalset}
	The global feasible set 
	\begin{equation*}
		\bm{\mc U} := \left( \textstyle \prod_{i\in\mc I} \mc L_i \right) \cap \bm{\mc C}\cap \bb R^{n_{x}+n_{y}}\times \{0,1\}^{n_{z}},
	\end{equation*}
	where {$\mc L_i$ is defined in \eqref{eq:Li} and} $$\bm{\mc C} := \{\bm u \in \bb R^n \mid \text{\eqref{eq:p_mg_bound}, \eqref{eq:pf_all}, } %\\
	\text{\eqref{eq:dtg_bound},  \eqref{eq:hf_cons2}, and \eqref{eq:gf_cons2}} \},$$ 
	is nonempty. \eod 
\end{assumption}

The game in \eqref{eq:ED_game1} is a generalized potential game \cite[Def. 2.1]{facchinei11decomposition}. To see this, let us denote by $\Xi_i^{\mathrm{e}}, \Xi_i^{\mathrm{g}} \in \bb R^{H \times Hn_{x_i}}$, for each $i \in \mc I$, the matrices that select $p_i^{\mathrm{eg}}$ and $d_i^{\mathrm{gu}}$ from $x_i$, i.e., $p_i^{\mathrm{eg}} = \Xi_i^{\mathrm{e}} x_i $ and $d_i^{\mathrm{gu}}=\Xi_i^{\mathrm{g}}x_i$, {and define} $Q^{\mathrm{e}} = \diag((q_h^{\mathrm{ e}})_{h\in \mc H})$ and $Q^{\mathrm{g}} = \diag((q_h^{\mathrm{ g}})_{h\in \mc H})$. {Furthermore, we let} $D_i = (\Xi_i^{\mathrm e})^\top Q^{\mathrm e}  \Xi_i^{\mathrm e} + (\Xi_i^{\mathrm g})^\top Q^{\mathrm g}  \Xi_i^{\mathrm g}$. %Then, the following statement holds.
\begin{lemma}
	\label{le:potential_game}
	Let Assumption \ref{as:nonempty_globalset} hold.  Then the game in \eqref{eq:ED_game1} is  a generalized potential game {\cite[Def. 2.1]{facchinei11decomposition}} with an exact potential function 
	\begin{equation}
		\label{eq:P}
		P(\bm x) = \tfrac{1}{2} \sum_{i\in\mathcal{I}}(J_i(\bm x) + f_i^{\mathrm{loc}}(x_i) +x_i^\top D_i x_i) ).
	\end{equation}
 % \eod
\end{lemma}
%\begin{proof}
%See Appendix \ref{app:pf:le:potential_game}.
%\end{proof}

We observe  {that $P$ in \eqref{eq:P} is convex and that} the pseudogradient  {of the game is monotone,} as formally stated next.
\begin{lemma}
	\label{le:mon_pseudograd_P_conv}
	{Let $J_i$ be defined as in \eqref{eq:Ji}.} The following statements hold:
	\begin{enumerate}[(i)]
		\item \label{le:mon_pseudogradient}
		 The mapping $	\col\left(\left(	\nabla_{x_i} J_i(\bm x) \right)_{i \in \mc I}
		\right)$ and thus, the pseudogradient mapping of the game in \eqref{eq:ED_game1},
		\begin{equation}
			F(\bm u) := 
%			\begin{bmatrix}
			{	\col\left(\left(	\nabla_{x_i} J_i(\bm x)\right)_{i \in \mc I}, \0 \right),}
		\end{equation}
		are monotone.
		\item \label{le:P_convex}
		{The potential function $P$ in \eqref{eq:P} is convex.} \eod 
	\end{enumerate}
\end{lemma}
%\begin{proof}
%	See Appendix \ref{app:pf:le:mon_pseudogradient}.
%\end{proof}
%Therefore, by Lemmas \ref{le:potential_game} and \ref{le:mon_pseudogradient}, the potential function $P$ in \eqref{eq:P} is convex. 

%\paragraph*{v-GNE as solution}
{In this paper, we postulate that} the objective of each {player (prosumer)} is to compute  an approximate generalized Nash equilibrium (GNE),  which is formally defined in Definition \ref{def:GNE}. 
\begin{definition}
	\label{def:GNE}
	A set of {strategies} $\bm u^{\ast} := (\bm x^\ast, \bm y^\ast, \bm z^\ast)$ is an $\varepsilon$-approximate GNE ($\varepsilon$-GNE) of the game in \eqref{eq:ED_game1} if,  $\bm u^\ast \in {\bm{\mc U}}$, and, there exists $\varepsilon \geq 0$ such that, for each $i \in \mc I$, 
	\begin{equation}
			 J_i(x_i^\ast,\bm x_{-i}^\ast)  \leq  J_i(x_i,\bm x_{-i}^\ast) + \varepsilon, \label{eq:eps_GNE}
	\end{equation}
	for any $u_i \in \mc L_i \cap \mc C_i(\bm u_{-i}^\ast)\cap \bb R^{n_{x_i}+n_{y_i}}\times \{0,1\}^{n_{z_i}}$, where
	%\begin{equation*}
	$\mc C_i(\bm u_{-i})	:=\{u_i \in \bb R^{n_{i}} \mid {(u_i, \bm u_{-i}) \in \bm{\mc C} }\}$, {and $\bm u_{-i}$ denotes the decision variables of all agents except agent $i$.}  {When $\varepsilon=0$, an $\varepsilon$-GNE is an exact GNE.} \eod 
	%\end{equation*}
\end{definition}

\iffalse 
\begin{remark}
	{As we consider two gas flow models, the mixed-integer ED games formed by using the two models are different and, consequently, may have different sets of $\varepsilon$-GNEs. }
\end{remark}

A particular subset of GNE that we are interested in is that of variational GNEs (v-GNEs), which has fairness and stability characteristics \cite{bibid}. Roughly speaking, at a v-GNE, each agent is equally penalized to satisfy the coupling constraints. Furthermore, since the game in \eqref{eq:ED_game1} is potential,  a v-GNE also corresponds to a solution of the following optimization:
\begin{equation}
\begin{aligned}
		\min_{\bm x, \bm y, \bm z} \quad &\sum_{i\in\mathcal{I}} P(\bm x) \\
		\operatorname{s.t.} \quad & u_i \in \mc L_i, \ z_i \in \{0,1\}^{n_{z_i}}, \quad \forall i \in \mc I, \\
		& \text{\eqref{eq:p_mg_bound}, \red{\eqref{eq:rep_const},} \eqref{eq:dtg_bound},  \eqref{eq:hf_cons2}, and \eqref{eq:gf_cons2},}
\end{aligned}
\label{eq:soc_welfare}
\end{equation}
where $P(\bm x)$ is defined in Proposition \ref{prop:potential_game}. By this 
%\paragraph*{relationship between v-GNE and social welfare}
\fi

\section{Two-stage equilibrium seeking approach}
\label{sec:prop_approach}
%{The existing GNE-seeking methods discussed in \cite{sagratella17,sagratella19,cenedese19,fabiani20} cannot deal with the game in \eqref{eq:ED_game1} as the Jacobi-type algorithms \cite{sagratella17,cenedese19,fabiani20} cannot handle the equality coupling constraints whereas the branch-and-bound method in \cite{sagratella19} is not designed for a generalized game whose coupling constraints involve continuous variables. Therefore,} 
{One way to find} an $\varepsilon$-GNE of the game in \eqref{eq:ED_game1} is {by computing an} $\varepsilon$-approximate global solution to the following problem \cite[Thm. 2]{sagratella17}:
\begin{equation}
\begin{cases}
		\underset{\bm u}{\min} & P(\bm x) \\
		\operatorname{s.t.} & \bm u \in \bm{\mc U},
\end{cases}
\label{eq:opt_P}
\end{equation}
{which exists due to Assumption \ref{as:nonempty_globalset}.} 
However, solving \eqref{eq:opt_P} in a centralized manner may be prohibitive since the problem can be very large, $\bm{\mc U}$ is non-convex, and prosumers might not want to share their private information such as their cost parameters and local constraints.

{Therefore,} we propose a two-stage  {distributed} approach to solve the generalized mixed-integer game in \eqref{eq:ED_game1}. In the first stage, we consider a convexified version of the MI game in \eqref{eq:ED_game1}, which can be solved by distributed equilibrium seeking algorithms based on operator splitting methods. In the second stage, each agent  {$i \in \mc I$} recovers the binary solution $z_i$ and finds the pressure variable $\psi_i$ that minimizes  the error of {the gas flow model} {by cooperatively solving a linear program.}

\subsection{Problem convexification}
\label{sec:conv_game}

Let us convexify the game in \eqref{eq:ED_game1}, by considering $z_i \in [0,1]^{n_{z_i}}$, for all $i \in \mc I$, i.e., the variable $z_i$ is continuous instead of discrete:
\begin{subequations}
	\begin{empheq}[left=\forall i \in \mc I \colon \empheqlbrace]{align}
		\label{eq:cost_gen1}	
		& 
		\underset{{x}_i, y_i, z_i}{\min}  \; \ \  J_{i}\left(
		\bm x
		\right) \\[.2em]
		\label{eq:const_gen2}
		&\quad \text{s.t. } \; \quad ({x}_{i},y_i,z_i) \in  \mc L_i, z_i \in [0,1]^{n_{z_i}}, \\
		&\qquad \quad \; \; \text{\eqref{eq:p_mg_bound}, \eqref{eq:pf_all}, \eqref{eq:dtg_bound},  \eqref{eq:hf_cons2}, and \eqref{eq:gf_cons2}.}% \\
		%
		%	& \qquad \qquad
		%	u_{i} \in \mc C_i(u_{-i})	
	\end{empheq}
	\label{eq:ED_game_c}%
\end{subequations} 
%\paragraph{characteristics of cost and constraints (satisfying aggregative games in \cite{belgioioso19}).}
By construction, the game in \eqref{eq:ED_game_c} is jointly convex \cite[Def 3.6]{facchinei10}. %Specifically, it falls to the class of aggregative game \cite[Def ?]{bibid}. 
Additionally, by Lemma \ref{le:mon_pseudograd_P_conv}.(\ref{le:mon_pseudogradient}), the game has a monotone pseudogradient. Therefore, one can choose a semi-decentralized  or distributed algorithm to compute an exact GNE, e.g. \cite{belgioioso20semi,franci20}. These algorithms specifically compute a variational GNE, i.e., a GNE where each agent is penalized equally in meeting the coupling constraints. In our case, a variational GNE is also a minimizer of the potential function over the convexified global feasible set, i.e.,
\begin{equation}
	\begin{cases}
		\underset{\bm u}{\min} & P(\bm x) \\
		\operatorname{s.t.} & \bm u \in \operatorname{conv}(\bm{\mc U}), 
	\end{cases}
	\label{eq:opt_P_cvx}
\end{equation}
where $\operatorname{conv}(\bm{\mc U}):=\left(\prod_{i\in\mc I} \mc L_i \right)\cap \bm{\mc C}\cap \bb R^{n_{x}+n_{y}}\times [0,1]^{n_{z}}$ denotes the convex hull of $\bm{\mc U}$.  Note in fact that the Karush-Kuhn-Tucker {optimality} conditions of a v-GNE of the game in \eqref{eq:ED_game_c} and that of Problem \eqref{eq:opt_P_cvx} coincide.  For the next stage of the method, let us denote  the equilibrium computed in this stage by $\hat{\bm{u}}= (\hat{\bm x}, \hat{\bm y}, \hat{\bm z})$, where $\hat{\bm x} = \col(\{\hat x_i\}_{i \in \mc I})$ {($\hat{\bm y}$ and $\hat{\bm z}$ are defined similarly).} % $\hat{\bm y} = \col(\{\hat y_i\}_{i \in \mc I})$, and $\hat{\bm z} = \col(\{\hat z_i\}_{i \in \mc I})$. 
%As an example of {a v-GNE-seeking algorithm for the game in \eqref{eq:ED_game_c},} we can use a semi-decentralized preconditioned proximal point method as stated in Algorithm \ref{alg:ppp} in Appendix \ref{ap:ppp}.

\subsection{Recovering binary decisions}
\label{sec:recompute_binary}

 In the first stage, {since} we relax the integrality constraint,  we cannot guarantee that we obtain binary solutions of $\hat{\bm z}$. {In fact,} if $\hat z_i \in \{0,1\}^{n_{z_i}}$, for all $i \in \mc I$, then $\hat{\bm u}$ is {an exact} GNE of the mixed-integer game in \eqref{eq:ED_game1}. However, we can recover binary solutions via the logical implications of the computed pressure and flow decisions.

For both gas-flow models, we recall that the variable $\delta_{(i,j),h}$ is used to indicate the flow direction in the link $(i,j)$ (according to \eqref{eq:phi_logic1} in Appendix \ref{ap:pwa_gf}). {Thus,} given $\hat \phi_{(i,j),h}$, we recompute $\delta_{(i,j),h}$ as follows:
\begin{equation}
	\tilde \delta_{(i,j),h}  = \begin{cases}
		1, \quad \text{if } \hat \phi_{(i,j),h} \geq 0, \\
		0, \quad \text{otherwise}, 
	\end{cases}
\label{eq:solve_binary1}
\end{equation}
for all $h \in \mc H$, $ j \in \mc N_i^{\mathrm{g}},$, and $i \in \mc I$. 
In addition, for the PWA model, the rest of the binary variables $\{\tilde \alpha_{(i,j)}^m, \tilde \beta_{(i,j)}^m, \tilde \gamma_{(i,j)}^m\}_{r=1}^m \}_{j \in \mc N_i^{\mathrm{g}}}$, which determine the active regions in which the flow decisions are, can be recovered via \eqref{eq:ub_cons1} as follows:
%\begin{equation}
	\begin{align}
		&\tilde 	\gamma_{(i,j),h}^m = \begin{cases}
			1, \quad \text{if } \underline{\phi}_{(i,j)}^m \leq \hat \phi_{(i,j),h} \leq \overline{\phi}_{(i,j)}^m, \\
			0, \quad \text{otherwise}, \ \  m=1,\dots, r, \forall j \in \mc N_i^{\mathrm{g}}, \forall h \in \mc H,
		\end{cases} \notag \\%\label{eq:phi_logic2b} \\
		&{-\tilde  \alpha_{(i,j),h}^m} +\tilde  \delta_{(i,j),h}^m \leq 0, \  % \label{eq:prod_cons1}%\\
		-\tilde \beta_{(i,j),h}^m +\tilde  \delta_{(i,j),h}^m\leq 0,  %\label{eq:prod_cons2}
		\notag \\ 
		&\tilde \alpha_{(i,j),h}^m +\tilde  \beta_{(i,j),h}^m -\tilde  \delta_{(i,j),h}^m\leq 1, \quad \forall j \in \mc N_i^{\mathrm g}, \forall h \in \mc H.%\label{eq:prod_cons4} 	
		\label{eq:solve_binary}
	\end{align}

%\end{equation}

\color{black}

\subsection{{Recovering solutions of the MISOC model}}
\label{sec:recompute_partialsol}

	Let us now focus on the formulation that uses the MISOC model  and discuss the approach for the PWA {one} later. Since $\tilde \delta_{(i,j)}$, for all $(i,j) \in \mc E^{\mathrm g}$, are binary decisions obtained from Section \ref{sec:recompute_binary}, the constraints of the MISOC model in \eqref{eq:gf_MIS2}--\eqref{eq:gf_MIS5} (Appendix \ref{ap:soc_gf}) are equivalent to $\nu_{(i,j),h} = (2\tilde{\delta}_{(i,j),h}-1)(\psi_{i,h} - \psi_{j,h}),$ for all $(i,j)\in \mc E^{\mathrm g}$ and $h \in \mc H$. Therefore, the relaxed SOC gas-flow constraint in \eqref{eq:gf_MIS1} of Appendix \ref{ap:soc_gf} can be written as
	\begin{equation}
		(2\tilde{\delta}_{(i,j),h}-1)(\psi_i - \psi_j) \geq \phi_{(i,j),h}^2/c_{i,j}^2,  %\forall h \in \mc H, \forall (i,j) \in \mc E^{\mathrm g}.
		\label{eq:gf_MIS6}
	\end{equation}
	for all  $h \in \mc H$ and  $ (i,j) \in \mc E^{\mathrm g}$. 
	When  we set $\psi_i = \hat \psi_i$, $\psi_j = \hat \psi_j$, and $\phi_{(i,j)}= \hat \phi_{(i,j)}$,  \eqref{eq:gf_MIS6} might not hold since $\tilde{\delta}_{(i,j)}$ can be different from $\hat{\delta}_{(i,j)}$, which is possibly not an integer. 
	
	Therefore, our next step is to recompute the pressure variables $\psi_i$, for all $i \in \mc I$. To that end, let us  first compactly write the pressure variable $\bm \psi =  \col((\bm \psi_h\}_{h \in \mc H})$ and $\bm \psi_h = \col((\psi_{i,h})_{i \in \mc I})$, the binary variables $\tilde{\bm{\delta}} = \col((\tilde{\bm{\delta}}_h)_{h \in \mc H})$, $\tilde{\bm{\delta}}_h = \col((\tilde{{\delta}}_{(i,j),h})_{(i,j)\in \mc E^{\mathrm g}})$,  and the flow variables $\hat{\bm{\phi}} = \col((\hat{\bm{\phi}}_h)_{h \in \mc H})$ and $\hat{\bm{\phi}}_h = \col((\hat{{\phi}}_{(i,j),h})_{(i,j)\in \mc E^{\mathrm g}}) $. Let us now define $\bm E(\tilde{\bm{\delta}}) := \blkdiag((E(\tilde{\bm{\delta}}_h))_{h \in \mc H})$, where $E(\tilde{\bm{\delta}}_h) := \col((E_{(i,j)}(\tilde \delta_{(i,j),h}))_{(i,j)\in \mc E^{\mathrm g}}) \in \bb R^{|\mc E^{\mathrm{g}}| \times N}$ has the structure of the incidence matrix of   $\mc G^{\mathrm{g}}$, since for each row $(i,j)$ of $E(\tilde{\bm{\delta}}_h)$, 
	\begin{equation}
		[E_{(i,j)}(\tilde \delta_{(i,j),h})]_k = 
		\begin{cases}
			(2 \tilde \delta_{(i,j),h}-1) \quad &\text{if } k=i, \\
			- (2 \tilde \delta_{(i,j),h}-1) \quad &\text{if } k = j,\\
			0 \quad & \text{otherwise},
		\end{cases}
		\label{eq:E_row}
	\end{equation}
	where $[E_{(i,j)}]_k$ denotes the $k$-th component of the (row) vector $E_{(i,j)}$, and since $\tilde \delta_{(i,j),h} \in \{0,1\}$. Thus, $\bm E(\tilde{\bm{\delta}}) \bm \psi$ equals to the concatenation of the left-hand side of \eqref{eq:gf_MIS6} over all $(i,j) \in \mc E^{\mathrm g}$ and $h \in \mc H$. 
	Furthermore, let us define the vector ${\bm \theta}(\hat{\bm{\phi}}):= \col(((\phi_{(i,j),h}^2/c_{(i,j)}^2)_{(i,j)\in \mc E^{\mathrm g}})_{h \in \mc H})$ that concatenates the right-hand-side of \eqref{eq:gf_MIS6}.  
	We can then formulate an optimization problem that minimizes not only the violation of the above gas flow inequality constraints \eqref{eq:gf_MIS6} but also {of} the original gas flow equation \eqref{eq:gf_eq}. By defining a slack vector $\bm{\tau}:=\col((\bm \tau_h)_{h \in \mc H})$, where $\bm \tau_h :=\col(\{\tau_{(i,j),h}\}_{(i,j) \in \mc E^{\mathrm g}})$, the desired optimization problem parameterized by $\hat{\bm \phi}$ and $\tilde{\bm \delta}$ is given as follows:
	\begin{subequations}
		\begin{empheq}[left= {(\tilde{ \bm \psi}, \tilde{ \bm \tau}) \in} \empheqlbrace]{align}
			\underset{\bm \psi, \bm \tau}{\min}  \ &  \|\bm \tau \|_{\infty} + J_\psi(\bm \psi)  \label{eq:qp_pressure_cost2}\\		
			\operatorname{s.t.}  \quad  & \bm E(\tilde{\bm{\delta}}) \bm \psi + \bm \tau \geq {\bm \theta}(\hat{\bm{\phi}}), \label{eq:gf_MIS7}\\
			& \bm \psi \in [ \underline{\bm \psi}, \overline{\bm \psi}], \ \ \bm \tau \geq 0, \label{eq:psi_limits2} 
		\end{empheq}
		\label{eq:qp_pressure2}%
	\end{subequations} 
	where the cost function 
	\begin{equation}
		J_\psi(\bm \psi):= \| \bm E(\tilde{\bm{\delta}}) \bm \psi - {\bm \theta}(\hat{\bm{\phi}}) \|_\infty \label{eq:Jpsi_tilde}
	\end{equation}
	indicates the maximum violation of the original gas flow equation in \eqref{eq:gf_eq}.
	  In this problem, we relax \eqref{eq:gf_MIS6} into \eqref{eq:gf_MIS7} to ensure the existence of feasible points of Problem \eqref{eq:qp_pressure2} and indeed $\bm \tau$ indicates the level of violation of \eqref{eq:gf_MIS6}.  Moreover, we introduce \eqref{eq:Jpsi_tilde} to induce a solution that satisfies the Weymouth equation. In addition, the constraints on $\bm \psi$ in \eqref{eq:psi_limits2} is obtained from \eqref{eq:pres_lim}, where $\underline{\bm \psi} = \1_{|\mc H|}\otimes\col(( \underline{\psi}_i)_{i \in \mc I})$ and $\overline{\bm \psi} = \1_{|\mc H|}\otimes\col((\overline{\psi}_i)_{i \in \mc I})$. %Similarly to Problem \eqref{eq:qp_pressure}, Problem \eqref{eq:qp_pressure2} is a multi-agent convex optimization that can be solved in a distributed manner.
	
	\color{black}
	\begin{remark}
		{The} problem  {in} \eqref{eq:qp_pressure2} can be either solved centrally or distributedly, e.g. by  resorting to the dual decomposition approach \cite{falsone17,ananduta22b}, since it can be equivalently written as a {multi-agent} linear optimization problem. \eod 
	\end{remark}

	Next, by using the pressure component of a solution to Problem \eqref{eq:qp_pressure2}, $\tilde{\bm \psi}$, the auxiliary variables of the MISOC model, i.e. $\bm \nu =((\nu_{(i,j)})_{(i,j) \in \mc E^{\mathrm g}})$, are updated as follows:
	\begin{align}
		\tilde \nu_{(i,j),h} = (2\tilde{\delta}_{(i,j),h}-1)(\tilde \psi_{i,h} - \tilde \psi_{j,h}), %\quad \forall h \in \mc H,
		\label{eq:aux_var_update2}
	\end{align}
	for all $h \in \mc H$, $j \in \mc N_i^{\mathrm g}$ and $i \in \mc I$.
	
	Let us now characterize the recomputed solution.
	\begin{proposition}
		\label{prop:0opt_is_MIGNE2}
		Consider the ED game in \eqref{eq:ED_game1} where the gas flow constraints \eqref{eq:hf_cons2}--\eqref{eq:gf_cons1} is defined based on the MISOC model in Section \ref{sec:gf_approx}.a.	Let $\hat{\bm{u}}= (\hat{\bm x}, \hat{\bm y}, \hat{\bm z})$ be a feasible point of the convexified ED game as in \eqref{eq:ED_game_c}, $\tilde{\bm z}$ satisfy \eqref{eq:solve_binary1} given $\hat{ \bm \phi}$, $(\tilde{\bm \psi}, \tilde{\bm \tau})$ be a solution to Problem \eqref{eq:qp_pressure2}, {and $\tilde{\bm \nu}$ satisfy \eqref{eq:aux_var_update2}.}  Furthermore, let us define ${\bm{u}}^\star= ({\bm x}^\star, {\bm y}^\star, {\bm z}^\star)$, where $\bm x^\star = \hat{\bm x}$, $ {\bm z}^\star = \tilde{\bm z}$, and $\bm y^\star = \col((y_i^\star)_{i \in \mc I})$, where $y_i^\star := \col(\tilde \psi_i, \hat g_i^{\mathrm{s}}, (\hat \phi_{(i,j)}, \tilde \nu_{(i,j)} )_{j \in \mc N_i^{\mathrm{g}}} )$. {The following statements hold:}
		\begin{enumerate}[(i)]
			\item The strategy ${\bm{u}}^\star$ is a feasible point of the ED game in \eqref{eq:ED_game1} if and only if  $ \tilde{\bm \tau} = \0$. 
			\item The strategy ${\bm{u}}^\star$ is a GNE of the ED game in \eqref{eq:ED_game1} if  $ \tilde{\bm \tau} = \0$ and $\hat{\bm{u}}$ is a {variational} GNE of the convexified ED game in \eqref{eq:ED_game_c}.
			\item The strategy ${\bm{u}}^\star$ satisfies the original gas flow equations in \eqref{eq:gf_eq} if $ J_\psi(\tilde{\bm \psi})=0$. \eod 
		\end{enumerate}
	\end{proposition}

\color{black}

{Based on Proposition \ref{prop:0opt_is_MIGNE2}.(ii), ideally we wish to find  a solution to Problem \eqref{eq:qp_pressure2} such that $ \tilde{ \bm \tau} = \0$,} {because that} means that we find an exact MI-GNE of the game in \eqref{eq:ED_game1}. However, in general, this might not be possible.  Nevertheless, we guarantee that the solution ${\bm{u}}^\star= ({\bm x}^\star, {\bm y}^\star, {\bm z}^\star)$, as defined in Proposition \ref{prop:0opt_is_MIGNE2}, satisfies all constraints but the gas-flow equations. Furthermore, $\bm u^\star$ is an approximate solution with minimum violation and {the value of} $\| \tilde{ \bm \tau}\|_{\infty}$ quantifies the {maximum} error. 
Additionally, 
if a  solution to Problem \eqref{eq:qp_pressure2} has $J_\psi(\tilde{\bm \psi})=0$, then essentially it solves the following linear equations:
\begin{equation}
	E(\tilde{\bm{\delta}}_h) \bm \psi_h -\bm \theta_h(\hat{\bm{\phi}_h})=0, \quad \forall h \in \mc H.
	\label{eq:lin_eq_psi_h2b}
\end{equation}
We {note} a necessary condition on the structure of the gas network that allows us to have a solution to the system of linear equations \eqref{eq:lin_eq_psi_h2b}, consequently, tight SOC gas flow constraints.
\begin{lemma}
	\label{le:tree_G}
	Let ${\mc G}^{\mathrm g}$ be an undirected connected graph representing the gas network. Then, there exists a set of solutions to the systems of linear equations in \eqref{eq:lin_eq_psi_h2b} if and only if ${\mc G}^{\mathrm g}$ is a minimum {spanning} tree. \eod 
\end{lemma}

%{Next, we provide some analysis on Problem \eqref{eq:qp_pressure}.} 
%\paragraph{necessity of having minimum tree gas network}

	\subsection{Penalty-based outer iterations}
	
	%\paragraph{Outer iteration} 
	In this section, we extend the two-stage approach by having outer iterations that can find a feasible strategy, i.e., an MI solution that satisfies all the constraints including the gas-flow equations \eqref{eq:gf_eq6}. {According to the linear constraint in \eqref{eq:gf_MIS7}, a smaller  $\hat{\bm \phi}$ induces a smaller $\|\bm \tau\|_\infty$ since $\bm \theta \geq 0$ is quadratically proportional to $\hat{\bm \phi}$.} 
	To obtain a small enough $\hat{ \bm \phi}$, instead of solving the continuous GNEP \eqref{eq:ED_game_c}  {in} the first stage, we solve an approximate {(convexified)} problem where we introduce an extra penalty on the flow variables ${\bm \phi}$ to the cost function {of each prosumer:} 
	\begin{equation}
		\tilde J_i(\bm x,\{\phi_{(i,j)}\}_{j \in \mc N_i^{\mathrm g}}) :=	J_i(x_i,\bm x_{-i}) + \rho \sum_{j \in \mc N_i^{\mathrm{g}}} \| \phi_{(i,j)}\|_\infty,
		\label{eq:Ji_pen}
	\end{equation}
	for all $i \in \mc I$, where $\rho\geq 0$ is the weight of the penalty term.  Therefore, the approximate {jointly convex} GNEP is written as 
	\begin{equation}
	%	\begin{empheq}[left=\forall i \in \mc I \colon \empheqlbrace]{align}
		%	\label{eq:cost_gen1}	
			\forall i \in \mc I \colon  \begin{cases}
				& 
			\underset{{x}_i, y_i, z_i}{\min}  \; \ \  \tilde J_{i}(
			\bm x, \{\phi_{(i,j)}\}_{j \in \mc N_i^{\mathrm g}}	)
			 \\[.2em]
		%	\label{eq:const_gen2}
			&\quad \text{s.t. } \; \quad ({x}_{i},y_i,z_i) \in  \mc L_i, z_i \in [0,1]^{n_{z_i}}, \\
			&\qquad \quad \; \; \text{\eqref{eq:p_mg_bound}, \eqref{eq:pf_all}, \eqref{eq:dtg_bound},  \eqref{eq:hf_cons2}, and \eqref{eq:gf_cons2}.}% \notag % \\
			%
			%	& \qquad \qquad
			%	u_{i} \in \mc C_i(u_{-i})	
		\end{cases}
	%
		%\end{empheq}
		\label{eq:ED_game_c_pen}%
	\end{equation} 
	Since $\lim_{\rho \to 0} \tilde{J}_i = J_i$, the iterations are carried out to find the smallest $\rho$ such that the second stage results in {$\|\tilde{\bm \tau}\|_\infty =0$,} implying feasibility of the solution {(see Proposition \ref{prop:0opt_is_MIGNE2}.(ii)).} 
	
		\begin{algorithm}[t!]
		\caption{Iterative method for computing {an} approximate solution of the game in  \eqref{eq:ED_game1}.}
		\label{alg:2stage_meth_iterate}
		\textbf{Initialization.  } \\
		Set $\rho^{(1)} = 0$ and its bounds $\underline{\rho}^{(1)}=0$, $\overline{\rho}^{(1)}=\infty$.
		
		\textbf{Iteration ($\ell=1,2,\dots$)}\\
		%\vspace{-10pt}
		\textbf{Stage 1 (Computing a convexified game solution)}
		\begin{enumerate}
			\item  Compute $(\hat{\bm x}^{(\ell)}, \hat{\bm y}^{(\ell)}, \hat{\bm z}^{(\ell)})$, a variational GNE of  the approximate convexified game \eqref{eq:ED_game_c_pen}, where $\rho = \rho^{(\ell)}$.
			%\item  
			%\begin{enumerate}
		\end{enumerate}
	\textbf{Stage 2 (Recovering a mixed-integer solution)}
	\begin{enumerate}
		\setcounter{enumi}{1}
			\item 		Obtain the binary variable, $\tilde{\bm z}^{(\ell)}$, from $\hat{\bm \phi}^{(\ell)}$ via \eqref{eq:solve_binary1}.
			\item Compute the pressure variable, $\tilde{\bm \psi}^{(\ell)}$, and the slack variable, $\bm \tau^{(\ell)}$, by solving Problem \eqref{eq:qp_pressure2}, where $\tilde{\bm{\delta}} = (\tilde{\bm{\delta}}^{\psi})^{(\ell)}$  and $ \hat{ \bm \phi} = \hat{ \bm \phi}^{(\ell)}$.
			\item {Update the auxiliary variables, $\tilde {\bm \nu}^{(\ell)}$ via \eqref{eq:aux_var_update2} where $\tilde{\bm{\delta}} = (\tilde{\bm{\delta}}^{\psi})^{(\ell)}$ and $ \tilde{\bm \psi} = \tilde{\bm \psi}^{(\ell)}$.}
			
				\end{enumerate}
		\textbf{Solution and parameter updates}	
		\begin{enumerate}		
			\setcounter{enumi}{4}
			\item \label{it:sol_alg_2st} Update the solution $\tilde{\bm u}^{(\ell)}:=(\tilde{\bm x}^{(\ell)}, \tilde{\bm y}^{(\ell)}, \tilde{\bm z}^{(\ell)})$, where $\tilde{\bm x}^{(\ell)}=\hat{\bm x}^{(\ell)}$, $\tilde{\bm y}^{(\ell)} =\col((\tilde y_i^{(\ell)})_{i \in \mc I})$, and $\tilde{\bm y}^{(\ell)} $ is  defined as $\tilde{\bm y}^{(\ell)} =\col((\tilde y_i^{(\ell)})_{i \in \mc I})$, and {$\tilde y_i^{(\ell)} = \col(\tilde \psi_i^{(\ell)}, (\hat g_i^{\mathrm{s}})^{(\ell)}, (\hat \phi_{(i,j)}^{(\ell)}, \tilde \nu_{(i,j)}^{(\ell)} )_{j \in \mc N^{\mathrm{g}}} )$.}
		\item Update the bounds of the penalty weight:
		\begin{equation*}
			\begin{cases}
				\underline{\rho}^{(\ell+1)} = \rho^{(\ell)}, \quad
				\overline{\rho}^{(\ell+1)} =	\overline{\rho}^{(\ell+1)}, \quad &\text{If } {\|\bm \tau^{(\ell)} \|_{\infty} }>0, \\
				\underline{\rho}^{(\ell+1)} = \underline{\rho}^{(\ell)}, \quad
				\overline{\rho}^{(\ell+1)} =	{\rho}^{(\ell)}, \quad &\text{Otherwise}.
			\end{cases}
		\end{equation*}
		\item Update the penalty weight, i.e., choose
		$$	\rho^{(\ell+1)} \in (\underline{\rho}^{(\ell+1)},\overline{\rho}^{(\ell+1)} ).$$
	\end{enumerate}
\end{algorithm}
	
	Our proposed iterative method is  {described}  in Algorithm \ref{alg:2stage_meth_iterate}.  In Steps 6--7, $\rho$ increases when {$\|\bm \tau \|_{\infty} $} is positive and decreases otherwise, except for the case ${\|\bm \tau^{(\ell)} \|_{\infty}}= 0$, where $\rho^{(2)} = \rho^{(1)}= 0$. In the latter case, we find an exact GNE of the game in \eqref{eq:ED_game1} in the first iteration, i.e., $\varepsilon=0$.
	%\paragraph{introducing penalty term implies solution is approximate $\epsilon$-GNE}
	Next, we characterize the solution obtained by Algorithm \ref{alg:2stage_meth_iterate} as an $\varepsilon$-GNE,  {where $\varepsilon$ is expressed}  in Theorem \ref{thm:character_sol}. 
	%\subsection{GNE selection approach}
	\begin{theorem}
		\label{thm:character_sol}
		{Let us consider the ED game in \eqref{eq:ED_game1} where the gas flow constraints \eqref{eq:hf_cons2}--\eqref{eq:gf_cons1} is defined based on the MISOC model in Section \ref{sec:gf_approx}.a.}	Let Assumption \ref{as:nonempty_globalset} hold and let the sequence $(\tilde{\bm u}^{(\ell)}))_{\ell \in \bb N}$ be generated by (Step \ref{it:sol_alg_2st} of)  Algorithm \ref{alg:2stage_meth_iterate}. Suppose that there exists a {finite} $\overline{\ell} \in \argmin_\ell \{\rho^{ {(\ell)}} \ \operatorname{s.t.}\|\bm \tau^{(\ell)} \|_{\infty} = 0\}$. Then, the solution $\tilde{\bm u}^{(\overline{\ell})}$ is $\varepsilon$-GNE of the game in \eqref{eq:ED_game1} with
		\begin{equation}
			\label{eq:eps_bound}
			\varepsilon = P(\tilde{\bm x}^{(\overline{\ell})}) - P(\tilde{\bm x}^{(1)}),
		\end{equation}
		where $P$ is defined in \eqref{eq:P}. \eod 
	\end{theorem}

\begin{remark}
{	 In view {of} Proposition \ref{prop:0opt_is_MIGNE2}.(iii), {an} $\varepsilon$-GNE of the mixed integer ED game in \eqref{eq:ED_game1} computed by Algorithm \ref{alg:2stage_meth_iterate} that has $\tilde{J}_{\psi}=0$ is also an $\varepsilon$-GNE of the (continuous but non-+convex) ED game with Weymouth gas flow equations, i.e., when the constraints \eqref{eq:hf_cons2}--\eqref{eq:gf_cons1} of the ED game in \eqref{eq:ED_game1} is substituted with \eqref{eq:gf_eq}.} \eod 
\end{remark}
%%%%%%%%%%%%%%%%%%%%%%%%%%%%%%%%%%%%%

\subsection{Method adjustment for the PWA gas flow model}
In this section, we discuss how {to} modify our proposed method when we consider the PWA gas flow model. Following the approach we use on the MISOC model, given the flow decision computed from the first stage $\hat{\bm{\phi}}$ (Section \ref{sec:conv_game}) and the binary decisions $\tilde{{\delta}}_{(i,j)}$, $\tilde{{\gamma}}_{(i,j)}^m$, for $m=1,\dots, r$ and all $(i,j) \in \mc E^{\mathrm g}$ (as discussed in Section \ref{sec:recompute_binary}), we now recompute the pressure decision variable $\hat{\bm \psi}$ by minimizing the error of the PWA gas flow approximation  \eqref{eq:gf_eq4}, restated as follows:
\begin{multline}
	\sum_{m = 1}^r \gamma_{(i,j),h}^m ( a_{(i,j)}^m {\phi}_{(i,j),h} + b_{(i,j)}^m) \\= {(2 \delta_{(i,j),h}-1)} \psi_{i,h} - (2\delta_{(i,j),h}-1) \psi_{j,h}, \label{eq:gf_eq6}
\end{multline}
for each $h \in \mc H$, $j \in \mc N_i^{\mathrm{g}}$, and $i \in \mc I$. We can observe that although $(\hat{\bm y},\hat{\bm z})$ satisfies  \eqref{eq:gf_eq6}, for all $j \in \mc N_i^{\mathrm{g}}$ and $i \in \mc I$, $(\hat{\bm y},\tilde{\bm z})$ might not. 

By using the compact notations of $\bm \psi$, $\tilde{\bm \delta}$, $\hat{\bm \phi}$, as in Section \ref{sec:recompute_partialsol} as well as  $\tilde{\bm{\gamma}}= \col((\tilde{\bm{\gamma}}_h)_{h \in \mc H})$, where $\tilde{\bm{\gamma}}_h = \col((\tilde{{\gamma}}_{(i,j),h})_{(i,j)\in \mc E^{\mathrm g}})$, $\tilde{{\gamma}}_{(i,j),h}= \col((\tilde{{\gamma}}_{(i,j),h}^{m})_{m=1}^r)$, 
we recompute $\bm \psi$ by solving the following convex program: 
\begin{subequations}
	\begin{empheq}[left={\tilde{\bm \psi} \hspace{-2pt}\in \hspace{-2pt}} \empheqlbrace]{align}
		\underset{\bm \psi}{\argmin}  \ &  \tilde J_\psi(\bm \psi) \hspace{-2pt}:=\hspace{-2pt} \| \bm E(\tilde{\bm{\delta}}) \bm \psi -\tilde{\bm \theta}(\hat{\bm{\phi}}, \tilde{\bm{\gamma}}) \|_\infty \label{eq:qp_pressure_cost}\\		
		\operatorname{s.t.}  \quad  & \bm \psi \in [ \underline{\bm \psi}, \overline{\bm \psi}], \label{eq:psi_limits}
	\end{empheq}
	\label{eq:qp_pressure}%
\end{subequations} 
where the objective function $\tilde J_\psi$ is derived from the gas-flow equation \eqref{eq:gf_eq6} as we aim at minimizing its error. Specifically, we define  
$\bm E(\tilde{\bm{\delta}})$ as the concatenated matrix as in \eqref{eq:E_row}, $\tilde{\bm \theta}=\col((\bm \theta_h)_{h \in \mc H}) $ with $\tilde{\bm \theta}_h(\hat{\bm \phi}_h, \tilde{\bm \gamma}_h) = \col((\tilde \theta_{(i,j),h}(\hat{ \phi}_{(i,j),h}, \tilde{\gamma}_{(i,j),h}) )_{(i,j)\in \mc E^{\mathrm g}})$~and
%\begin{equation}
$	\tilde \theta_{(i,j),h}(\hat{ \phi}_{(i,j),h}, \tilde{\gamma}_{(i,j),h}) \hspace{-2pt}=\hspace{-2pt} \sum_{m = 1}^r\hspace{-2pt} \tilde \gamma_{(i,j),h}^m \hspace{-2pt} \left( \hspace{-2pt} a_{(i,j)}^m \hat{\phi}_{(i,j),h} \hspace{-2pt}+\hspace{-2pt} b_{(i,j)}^m\hspace{-2pt}\right)\hspace{-2pt},$ 
%\end{equation}
which is equal to the concatenation of the left-hand side of the equation in \eqref{eq:gf_eq6}. %In addition, the constraints in \eqref{eq:psi_limits} is obtained from \eqref{eq:pres_lim}, where $\underline{\bm \psi} = \1_{|\mc H|}\otimes\col(( \underline{\psi}_i)_{i \in \mc I})$ and $\overline{\bm \psi} = \1_{|\mc H|}\otimes\col((\overline{\psi}_i)_{i \in \mc I})$.

{
	Finally, we can update the auxiliary variable {$\bm{\nu}:=\col(((\nu_{(i,j)}^{\psi}, (\nu_{(i,j)}^m)_{m=1}^r )_{j \in \mc N_i^{\mathrm{g}}})_{i\in \mc I})$, where} $\nu_{(i,j)}^{\psi}$ and $\nu_{(i,j)}^m$ {satisfy} their definitions, i.e., 
	\begin{equation}
		\begin{aligned}
			\tilde \nu_{(i,j),h}^{\psi} &= \tilde{{\delta}}_{(i,j),h} \tilde{\psi}_{i,h}, \\
			\tilde \nu_{(i,j),h}^m &= \tilde \gamma_{(i,j),h}^m \hat \phi_{(i,j),h}, \ m=1,\dots,r,
		\end{aligned}
		\label{eq:aux_var_update}
	\end{equation}
	for all $h \in \mc H$, $j \in \mc N_i^{\mathrm g}$ and $i \in \mc I$.}
%\subsection{Analysis of Problem \eqref{eq:qp_pressure}}
%In this section, we analyze Problem \eqref{eq:qp_pressure} to 

%\paragraph{0 optimal solution implies MI-GNE} 
%Next, we characterize the recomputed solution.

Similarly to Proposition \ref{prop:0opt_is_MIGNE2}, when we consider the PWA model, the decision computed after performing the two stage is a variational GNE if $\tilde J_\psi(\tilde{\psi}) = 0$.
{Furthermore, $\tilde J_\psi(\tilde{\bm \psi})=0$ can be obtained only if $\mc G^{\mathrm g}$ does not have any cycle (c.f. Lemma \ref{le:tree_G}) since this condition guarantees the existence of solutions to the following  (linear) gas flow equations:}
\begin{equation}
	E(\tilde{\bm{\delta}}_h) \bm \psi_h -\bm \theta_h(\hat{\bm{\phi}_h}, \tilde{\bm{\gamma}})=0, \quad \forall h \in \mc H.
	\label{eq:lin_eq_psi_h2}
\end{equation}

Therefore, we can hope that the solution to Problem \eqref{eq:qp_pressure} has $0$ optimal value  {only} when Assumption \ref{as:tree_G} holds: 
\begin{assumption}
	\label{as:tree_G}
	The undirected graph $\mc G^{\mathrm g}$ that represents the gas network is a minimum spanning tree. \eod 
\end{assumption}

{In our case,} this assumption is acceptable as a distribution network typically has a tree structure.
{Note that,} Assumption \ref{as:tree_G} is only necessary. When it holds, the system of linear equations \eqref{eq:lin_eq_psi_h2} has infinitely many solutions. Let us suppose that $\bm \psi^0=\col(\{\bm \psi_h^0\}_{h \in \mc H})$ is a particular solution to \eqref{eq:lin_eq_psi_h2}, which does not necessarily satisfy the constraint in \eqref{eq:psi_limits}. One can compute such a solution by, e.g., {\cite[Ex. 29.17]{bauschke11}}:
\begin{equation}
	\bm \psi_h^0 =  (E(\tilde{\bm{\delta}}_h))^\dagger \bm \theta_h(\hat{\bm{\phi}}_h, \tilde{\bm{\gamma}}_h), \quad \forall h \in \mc H,
	\label{eq:psi_0}
\end{equation}
where $(\cdot)^\dagger$ denotes the pseudo-inverse operator. Given a particular solution $\bm \psi_0$, we can obtain necessary and sufficient conditions for obtaining a solution to Problem \eqref{eq:qp_pressure} with $0$ optimal value.

%\paragraph{sufficient and necessary conditions for existence of 0 optimal solution} \red{[Definition variables for Prop \label{sn_0sol}.]}
\begin{proposition}
	\label{prop:sn_0sol}
	Let Assumption \ref{as:tree_G} hold, $\bm \psi^0$ be a particular solution to \eqref{eq:lin_eq_psi_h2}, and $\tilde{\bm \psi}$ be a solution to Problem \eqref{eq:qp_pressure}. The optimal value $J_\psi(\tilde{\bm \psi})=0$ if and only if $$ [\bm \psi_h^0]_j -  [\bm \psi_h^0]_k \leq \overline{\psi}_{j} - \underline{\psi}_{k},$$ for all $h \in \mc H$, where $j \in \argmin_{\ell \in \operatorname{argmax}_{i \in \mc I} [\bm \psi_h^0]_i} \overline{\psi}_\ell$, and $k \in \operatorname{argmax}_{\ell \in \operatorname{argmin}_{i \in \mc I} [\bm \psi_h^0]_i}\overline{\psi}_\ell$. \eod 
\end{proposition}
%\begin{proof}
%	See Appendix \ref{app:pf:prop:sn_0sol}.
%\end{proof}

The condition given in Proposition \ref{prop:sn_0sol} cannot be checked a priori, that is, the condition depends on the outcome of the first stage. {Therefore, we can instead employ the iterative Algorithm \ref{alg:2stage_meth_iterate} for finding an $\varepsilon$-GNE. By considering the approximated problem \eqref{eq:ED_game_c_pen}, penalizing $\bm{\phi}$ induces a small-norm particular solution to \eqref{eq:lin_eq_psi_h2}, which is  required to be sufficiently small according to Proposition \ref{prop:sn_0sol}. 
We summarize the adaptation of Algorithm \ref{alg:2stage_meth_iterate} for the PWA model, as follows:}
	\begin{itemize}
	%	\item In stage 1 (Step 1), the cost functions of the game in \eqref{eq:ED_game_c_pen} are substituted with \eqref{eq:Ji_pen2}.
		\item In Step 2, we compute the binary variable $\tilde{\bm z}^{(\ell)}$  via \eqref{eq:solve_binary1}--\eqref{eq:solve_binary}.
		\item In Step 3, we compute the pressure variable $\tilde{\bm \psi}^{(\ell)}$ by solving Problem \eqref{eq:qp_pressure}.
		\item In Step 4, we obtain the auxiliary variables $\tilde{\bm \nu}$ via \eqref{eq:aux_var_update}.
		\item In Step 5, the updated vector $\tilde{\bm y}^{(\ell)} $ is  defined as $\tilde y_i^{(\ell)} = \col(\tilde \psi_i^{(\ell)}, (\hat g_i^{\mathrm{s}})^{(\ell)}, (\hat \phi_{(i,j)}^{(\ell)}, (\tilde \nu_{(i,j)}^{\psi})^{(\ell)}, ((\tilde \nu_{(i,j)}^m)^{(\ell)} )_{m=1}^r )_{j \in \mc N^{\mathrm{g}}} )$.
		\item In Step 6, the condition checked to update $\rho$ is whether $\tilde J_\psi > 0$. 
		\end{itemize} 
%\end{definition}
With this modification {and the addition of Assumption \ref{as:tree_G},} the characterization of $\varepsilon$-GNE of the solution computed by Algorithm \ref{alg:2stage_meth_iterate} considering the ED game with the PWA model is analogous to {that in} Theorem \ref{thm:character_sol}.

\color{black}
\section{Numerical simulations}
\label{sec:sim_res}

\begin{figure}
	\centering
	\includegraphics[scale=0.75]{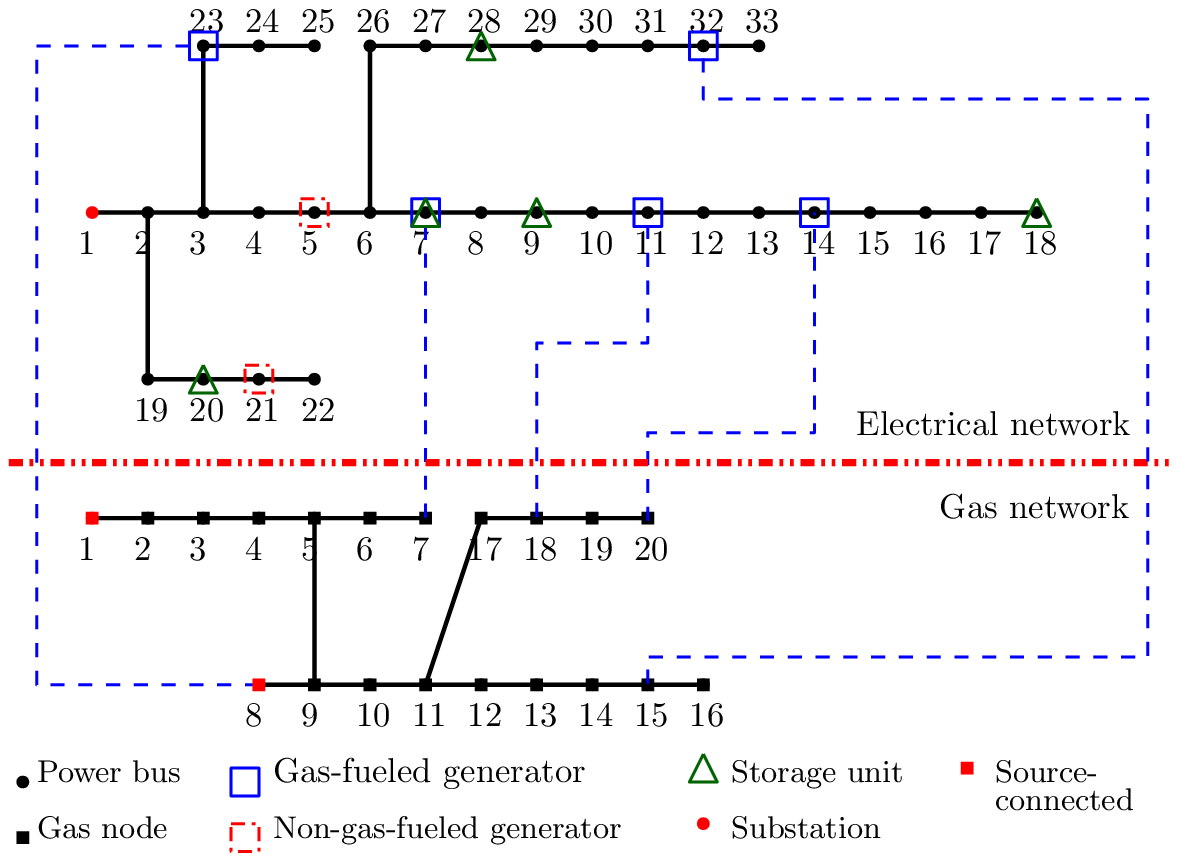}
	\caption{The 33-bus-20-node integrated electrical and gas distribution system. }
	\label{fig:testcase}
\end{figure}
 
 \iffalse 
 \begin{figure}
 	\centering
 	\includegraphics[width=\columnwidth]{imgs/sim1a}
 	\caption{The evolutions of the penalty weight $\rho^{(\ell)}$ (top), maximum error of gas-flow constraints $J_\psi^{(\ell)}$ (middle), and the value of the potential function $P^{(\ell)}$. }
 	\label{fig:sim1}
 \end{figure}
 \fi 

In {the following} numerical study, we aim at evaluating the performance of Algorithm \ref{alg:2stage_meth_iterate}. We use the 33-bus-20-node network\footnote{The data of the network and the codes are available at \texttt{https://github.com/ananduta/iegds}.} adapted from \cite{li18privacy}. An interconnection between the electrical and gas networks occurs when a prosumer (bus) has a gas-fired DG,  as illustrated in Figure \ref{fig:testcase}.  We generate {100} random test cases where some parameters, such as the gas loads, the locations of generation and storage units, as well as the interconnection points, vary. 
%{For each test case, we use PWA models with different numbers of regions ($r$), i.e., $r \in \{15,30,45,60,75\}$.  Thus, in total we have 500 instances of the mixed integer game.} 
    %First, we illustrate the iterations of Algorithm \ref{alg:2stage_meth_iterate}. For this simulation, we set the number of PWA regions to $r=30$.  The evolution of $\rho$, maximum error the gas-flow constraints $J_\psi$, and the value of potential function $P$ over the iteration $\ell$ can be seen in Figure \ref{fig:sim1}. By setting $\rho = 0.2181$, the algorithm finds an $\varepsilon$-GNE with $\varepsilon= 359.32$, which is on average $3.78\%$ of the costs of the prosumers.
%\com{Edit based on new results}
%Secondly, we evaluate how the PWA approximation method affects the performance of Algorithm \ref{alg:2stage_meth_iterate}. 
{For each test case, we use the MISOC model and two PWA models with two different numbers of regions, i.e. $r \in \{20,45\}$, thus in total we have 300 instances of the mixed integer game. We perform the simulations in Matlab on a computer with Intel Xeon E5-2637 3.5 GHz processors and 128 GB of memory.}
   \begin{figure}
	\centering
	\includegraphics[width=\columnwidth]{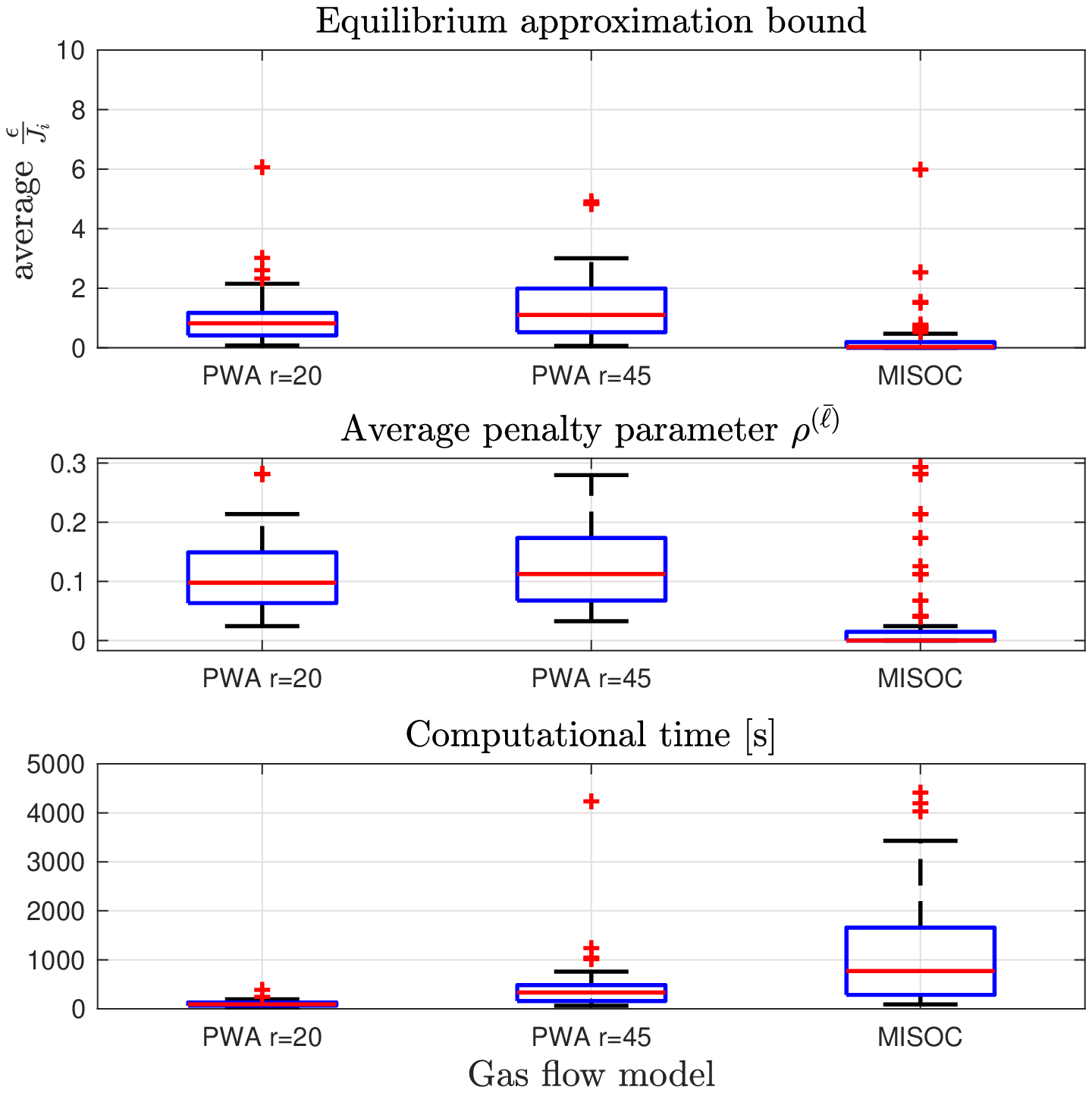}
	\caption{{Box plots of the percentage of average $\varepsilon/J_i$, the required penalty weight $\rho^{\overline{\ell}}$, and the computational time of Algorithm \ref{alg:2stage_meth_iterate} with different gas flow models. %Each bar and red line represent the mean and the median of the successful test cases. %average  gas-flow deviation ($\tfrac{1}{H |\mc E^{\mathrm g}|} \sum_{h=1}^H \sum_{(i,j) \in \mc E^{\mathrm g}} \Delta^\phi_{(i,j),h}$, where $\Delta^\phi_{(i,j),h}$ is defined in \eqref{eq:gf_dev})  obtained by Algorithm \ref{alg:2stage_meth_iterate} with different numbers of PWA regions ($r$). In the two top most plots, each data point represent the average value over 100 random test cases. In the bottom plot, 
	}}
	\label{fig:perform}
\end{figure}

 %In $40\%$ of the simulated instances, Algorithm \ref{alg:2stage_meth_iterate} obtains exact MI-GNEs, i.e., $\varepsilon=0$, while in others, it finds $\varepsilon$-approximate MI-GNEs. 
{Figure \ref{fig:perform} illustrates the performance of Algorithm \ref{alg:2stage_meth_iterate} with those gas flow models. We obtain the plots from the successful cases, i.e. when Algorithm \ref{alg:2stage_meth_iterate} finds an $\varepsilon$-MIGNE after at most 10 iterations, and they account for approximately $50\%$ of the generated cases. From the top plot of Figure \ref{fig:perform}, we can observe that Algorithm \ref{alg:2stage_meth_iterate} with the MISOC model produces  approximate solutions with the smallest $\varepsilon$. The penalty values, $\rho^{\overline{\ell}}$, required to obtain (approximate) solutions are shown in the middle plot. %Consequently, Algorithm \ref{alg:2stage_meth_iterate} requires a slightly higher penalty parameter  to obtain (approximate) solutions, $\rho^{\overline{\ell}}$, on the PWA models than the MISOC model, as  shown in the middle plot of Figure \ref{fig:perform}. 
	The bottom plot shows the average computational times of Algorithm \ref{alg:2stage_meth_iterate}. As expected, Algorithm \ref{alg:2stage_meth_iterate} needs longer time to find a solution on the PWA model with the larger $r$, as the number of decision variables grow proportionally with $r$. We can also observe that the MISOC model is more computationally demanding than the PWA model with $r=45$. }

  \begin{figure}
	\centering
	\includegraphics[width=\columnwidth]{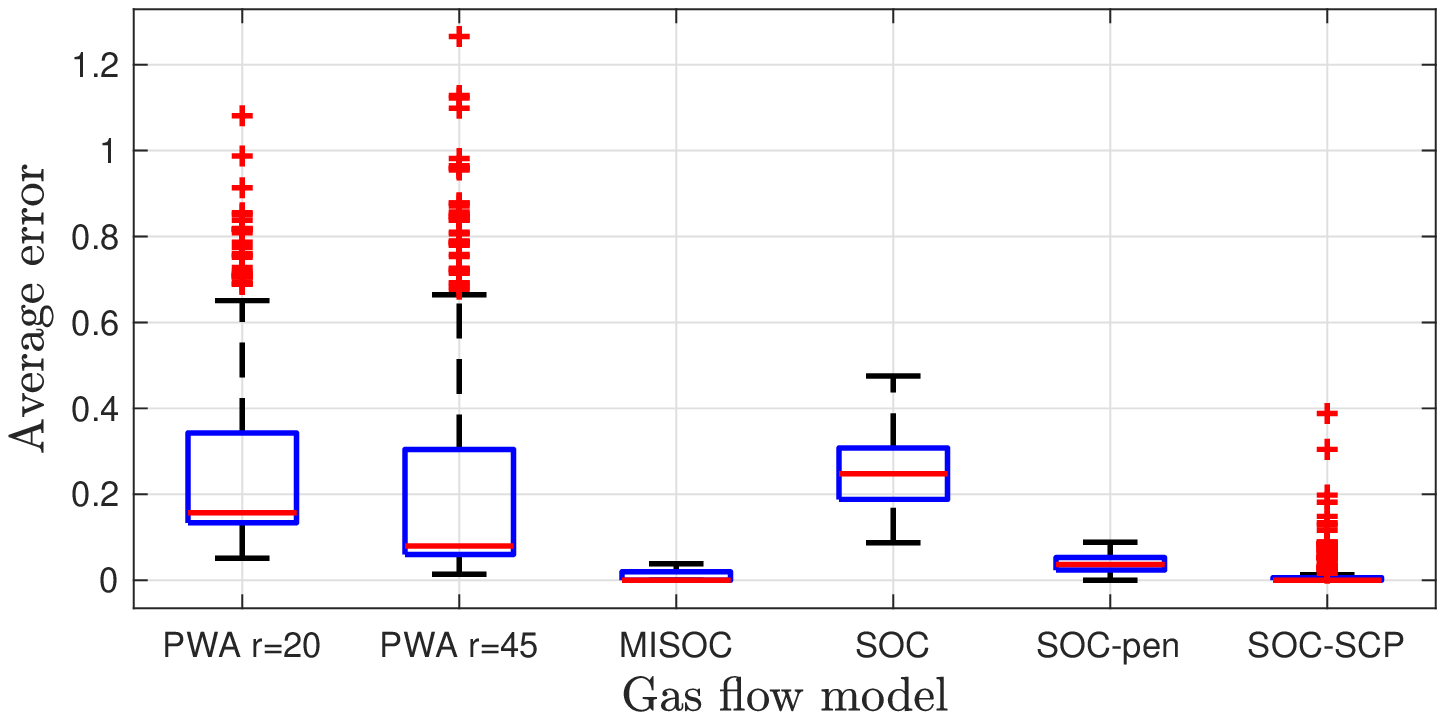}
	\caption{{ Box plots of the average  gas-flow deviations ($\tfrac{1}{H |\mc E^{\mathrm g}|} \sum_{h=1}^H \sum_{(i,j) \in \mc E^{\mathrm g}} \Delta^\phi_{(i,j),h}$, where $\Delta^\phi_{(i,j),h}$ is as in \eqref{eq:gf_dev}).
	}}
	\label{fig:gf_deviation}
\end{figure}
{Next, we compare the approximation quality of the gas flow models with respect to the Weymouth equation based on the deviation metric derived from \eqref{eq:gf_eq}, i.e., 
\begin{equation} \small 
	\Delta^\phi_{(i,j),h} = \frac{\hat \phi_{(i,j),h}- \operatorname{sgn}(\tilde \psi_{i,h}-\tilde \psi_{j,h}) c_{(i,j)}^{\mathrm f}\sqrt{|\tilde \psi_{i,h}-\tilde \psi_{j,h}|}}{ \operatorname{sgn}(\tilde \psi_{i,h}-\tilde \psi_{j,h}) c_{(i,j)}^{\mathrm f}\sqrt{|\tilde \psi_{i,h}-\tilde \psi_{j,h}|}},  \label{eq:gf_dev}
\end{equation}
for all $h \in \mc H$, $j \in \mc N_i^{\mathrm g}$, and $i \in \mc I$, as shown in Figure \ref{fig:gf_deviation}.   
As expected, for the PWA model,
  a larger $r$ implies a smaller deviation, hence a better approximation of the gas-flow constraints. However, the MISOC model outperforms the PWA models as Algorithm \ref{alg:2stage_meth_iterate} with the MISOC model can find a solution that satisfies the Weymouth equation in many of the simulated cases (the case of Proposition \ref{prop:0opt_is_MIGNE2}.(iii) holds). We also compare the two mixed integer models with the standard convex SOC relaxation model, i.e., the MISOC model with fixed binary variables (assuming known gas flow directions) \cite{wang18,chen20}, resulting in a jointly convex game. We note that we set the gas flow directions of the convex model based on the solutions of Algorithm \ref{alg:2stage_meth_iterate}. Moreover, we implement two approaches to reduce the gas flow deviation of the convex SOC model, namely adding a penalty cost on the auxiliary variable of the SOC model, $\bm \nu$, (SOC-pen) \cite{he18} and the iterative sequential cone program (SOC-SCP) \cite{wang18,he18}, which requires solving the game at each iteration, as in Algorithm \ref{alg:2stage_meth_iterate}. As observed in Figure \ref{fig:gf_deviation}, in terms of the gas flow deviation, the PWA models perform better than the convex SOC relaxation while the MISOC model outperforms the SOC-pen and performs as well as the SOC-SCP.}

\section{Conclusion}
The economic dispatch problem in integrated electrical and gas distribution systems can be formulated as a mixed-integer generalized potential game if the gas-flow equations are approximated by a {mixed-integer second order cone (MISOC) or} {by a} piece-wise affine  {(PWA)}  model. Our proposed algorithmic solution involves computing a {variational} equilibrium of the convexified game and leveraging the gas-flow model. { An a-posteriori characterization of the outcome of the algorithm is obtained via the evaluation of {the} potential function. Numerical simulations indicate that our algorithm, {particularly if paired with the MISOC model,} finds approximate equilibria {of very high} quality.}  

\appendices
\section{Gas flow approximation models}
{
	In this section, we provide the two commonly considered approximation models of gas-flow equations in \eqref{eq:gf_eq}. For ease of presentation and with a slight abuse of notation, we drop the time index $h$ of each variable.
}

\subsection{Mixed-integer second-order cone model}
\label{ap:soc_gf}
We summarize the mixed-integer second-order cone relaxation of the gas flow equations in \cite{he18} with a reduced number of binary variables. We use the binary variable $\delta_{(i,j)}$ to define the flow direction in the pipeline $(i,j)$, i.e.
\begin{align}
	[\delta_{(i,j)}  =1] &\Leftrightarrow [ \phi_{ij} \geq 0].  \label{eq:phi_logic1}
\end{align}
resulting in the following constraint:
\begin{equation}
	-\overline{\phi}_{(i,j)}(1-\delta_{(i,j)}) \leq {\phi}_{(i,j)} \leq \overline{\phi}_{(i,j)} \delta_{(i,j)}. \label{eq:phi_boundMIS}
\end{equation}
Consequently,  $\delta_{(i,j)}$ also indicates the pressure relationship between two connected nodes $i$ and $j$, i.e.
\begin{align}
	[\delta_{(i,j)}  =1] &\Leftrightarrow [ \psi_i \geq \psi_j], \label{eq:psi_logic}
\end{align}
Therefore, by squaring \eqref{eq:gf_eq} and including  $\delta_{(i,j)}$, we obtain an equivalent representation of the gas flow equation, as follows:
\begin{equation}
	\phi_{(i,j)}^2/c_{(i,j)}^{\mathrm f} =(2\delta_{(i,j)}-1) (\psi_{i}-\psi_{j}),  \label{eq:gf_eqMIS}
\end{equation}
for each $j \in \mc N_i^{\mathrm g}$ and $i \in \mc I$.

We  use an auxiliary variable, denoted by $\nu_{(i,j)} \in \bb R^H$, to substitute the right-hand-side of \eqref{eq:gf_eqMIS} and then relax \eqref{eq:gf_eqMIS} into a convex inequality constraint. Furthermore, by the McCormick envelope, we obtain linear relationships between $\nu_{(i,j)}$, $\delta_{(i,j)}$, $\psi_{i}$ and $\psi_{j}$. As a result, we obtain the following MISOC model:
\begin{align}
	\nu_{(i,j)} &\geq	\phi_{(i,j)}^2/c_{(i,j)}^{\mathrm f},  \label{eq:gf_MIS1}\\
	\nu_{(i,j)} &\geq \psi_j - \psi_i + 2\delta_{(i,j)} (\underline{\psi}_i - \overline{\psi}_j), \label{eq:gf_MIS2}\\
	\nu_{(i,j)} &\geq \psi_i - \psi_j + (2\delta_{(i,j)}-2) (\overline{\psi}_i - \underline{\psi}_j), \label{eq:gf_MIS3}\\	
	\nu_{(i,j)} &\leq \psi_j - \psi_i + 2\delta_{(i,j)} (\overline{\psi}_i - \underline{\psi}_j), \label{eq:gf_MIS4}\\	
	\nu_{(i,j)} &\leq \psi_i - \psi_j + (2\delta_{(i,j)}-2) (\underline{\psi}_i - \overline{\psi}_j), \label{eq:gf_MIS5}
\end{align}
for all $j \in \mc N_i^{\mathrm g}$ and $i \in \mc I$. 

Therefore, we can compactly represent the local constraints in \eqref{eq:phi_boundMIS} and \eqref{eq:gf_MIS1}, for all $j \in \mc N_i^{\mathrm g}$, as in \eqref{eq:gf_cons1} and the coupling constraints in \eqref{eq:gf_MIS2}--\eqref{eq:gf_MIS5}, for all $j \in \mc N_i^{\mathrm g}$, as in \eqref{eq:gf_cons2}. In addition, we also include the reciprocity constraints:
\begin{equation}
	\phi_{(i,j)}  + \phi_{(j,i)}  = 0, \quad \forall j \in \mc N_i^{\mathrm{g}},  i \in \mc N^{\mathrm{g}}, \label{eq:recip_flow}
\end{equation}
 for all $j \in \mc N_i^{\mathrm g}$, which can be rewritten as in \eqref{eq:hf_cons2}. For completeness, we define $h_i^{\mathrm{loc}}=0$.
\color{black}

\subsection{Piece-wise affine gas-flow model}
\label{ap:pwa_gf}	 
%\vspace{15pt}

%\paragraph*{piece-wise affine approximation of gas-flow equation}
In this section, we derive a {PWA model} of the gas-flow equation in \eqref{eq:gf_eq}. First, let us introduce an auxiliary variable  $\varphi_{(i,j)} := \tfrac{ \phi_{(i,j)}^2}{	(c_{(i,j)}^{\mathrm f})^2}$ and  rewrite \eqref{eq:gf_eq} as follows:
\begin{equation}
	\varphi_{(i,j)} = \begin{cases}
		(\psi_i - \psi_j) \quad \text{if } \psi_i \geq  \psi_j, \\
		(\psi_j - \psi_i) \quad \text{otherwise.}
	\end{cases}
	\label{eq:gf_eq2}
\end{equation}
{Similarly to the MISOC model,} we use the binary variable $\delta_{(i,j)} \in \{0,1\}$ {to define the flow directions as in \eqref{eq:phi_logic1} and \eqref{eq:psi_logic}.} Therefore, 
we can rewrite \eqref{eq:gf_eq2} as 
\begin{align}
	\varphi_{(i,j)} &= \delta_{(i,j)} (\psi_i - \psi_j) +(1-\delta_{(i,j)})(\psi_j - \psi_i )  \notag\\
	&=2 \delta_{(i,j)} \psi_i - 2\delta_{(i,j)} \psi_j + (\psi_j - \psi_i). \label{eq:gf_eq3}
\end{align} 
Now, we consider a PWA approximation of the quadratic function $\varphi_{(i,j)} = \tfrac{ \phi_{(i,j)}^2}{	(c_{(i,j)}^{\mathrm f})^2}$ {by partitioning} the operating region of the flow into $r$ subregions and {introducing} a binary variable 	$\gamma_{(i,j)}^m$, for each  {subregion} $m\in \{1,\dots,r\}$, defined by
\begin{align}
	[\gamma_{(i,j)}^m=1] \Leftrightarrow [\underline{\phi}_{(i,j)}^m \leq \phi_{(i,j)} \leq \overline{\phi}_{(i,j)}^m], \label{eq:phi_logic}
\end{align}
with $-\overline{\phi} = \underline{\phi}_{(i,j)}^1< \overline{\phi}_{(i,j)}^1=\underline{\phi}_{(i,j)}^2< \dots <\overline{\phi}_{(i,j)}^r=\overline{\phi}$. Then, we can consider the following approximation:
\begin{equation}
	\varphi_{(i,j)} \approx \sum_{m = 1}^r \gamma_{(i,j)}^m ( a_{(i,j)}^m {\phi}_{(i,j)} + b_{(i,j)}^m), \label{eq:pwa_approx_varphi}
\end{equation}
for some $a_{(i,j)}^m, b_{(i,j)}^m \in \bb R$, which can be obtained by using the upper and lower bounds of each subregion, i.e.,
\begin{align*}
	a_{(i,j)}^m =  \tfrac{(\overline{\phi}_{(i,j)}^m)^2-(\underline{\phi}_{(i,j)}^m)^2}{(c_{(i,j)}^{\mathrm f})^2(\overline{\phi}_{(i,j)}^m-\underline{\phi}_{(i,j)}^m)}, \
	b_{(i,j)}^m = (\underline{\phi}_{(i,j)}^m)^2-a_{(i,j)}^m \underline{\phi}_{(i,j)}^m,
\end{align*}
for $m =1, \dots, r$. 
By {approximating} $\varphi_{(i,j)}$ {with} \eqref{eq:pwa_approx_varphi}, we can then rewrite \eqref{eq:gf_eq3} as follows:
\begin{equation}
	\sum_{m = 1}^r \gamma_{(i,j)}^m ( a_{(i,j)}^m {\phi}_{(i,j)} + b_{(i,j)}^m) = 2 \delta_{(i,j)} \psi_i - 2\delta_{(i,j)} \psi_j + \psi_j - \psi_i. \label{eq:gf_eq4}
\end{equation}
Next,  we substitute the products of two variables with some auxiliary variables, i.e., $\nu_{(i,j)}^m:=\gamma_{(i,j)}^m{\phi}_{(i,j)}$, for $m=1,\dots,r$, and  $\nu_{(i,j)}^{\psi} =\delta_{(i,j)} \psi_i$. Furthermore, for $j \in \mc N_i^{\mathrm{g}}$, we observe that $ \delta_{(i,j)}= 1-\delta_{(j,i)},$ and 
$\delta_{(j,i)} \psi_j = \nu_{(j,i)}^{\psi}$,  {implying that}
$\delta_{(i,j)} \psi_j = ( 1-\delta_{(j,i)}) \psi_j = \psi_j - \nu_{(j,i)}^{\psi}.$ 
%\clearpage 
Thus, from \eqref{eq:gf_eq4}, we obtain the following gas-flow equation, for each $(i,j) 
\in \mc E^{\mathrm g}$:
\begin{equation} \small 
	\sum_{m = 1}^r  ( a_{(i,j)}^m \nu_{(i,j)}^m + b_{(i,j)}^m\gamma_{(i,j)}^m) = 2 \nu_{(i,j)}^{\psi} + 2\nu_{(j,i)}^{\psi} - \psi_i - \psi_j, \label{eq:gf_eq5}
\end{equation}
which is linear, involves binary and continuous variables, and couples the decision  {variables} of nodes $i$ and $j$. 
In addition to \eqref{eq:gf_eq5}, {we include \eqref{eq:phi_boundMIS}, \eqref{eq:recip_flow}, and} the following constraints:
\begin{equation}
	\sum_{m=1}^r \gamma_{(i,j)}^m = 1,  \label{eq:deltam}
\end{equation}
since only one subregion can be active;
\begin{align}
	\begin{cases}
		- \psi_i+\psi_j  \leq -(\underline \psi_i-\overline \psi_j)(1-\delta_{(i,j)}), \\
		- \psi_i+\psi_j \geq  (-(\overline \psi_i-\underline \psi_j) ) \delta_{(i,j)}, %\label{eq:psi2}%
	\end{cases}
	\label{eq:psi1}
\end{align}
which are equivalent to the logical constraint in \eqref{eq:psi_logic}; 
\iffalse 
\begin{align}
	\begin{cases}
		-\phi_{(i,j)} \leq \overline{\phi}_{(i,j)} (1-\delta_{(i,j)}), \\
		-\phi_{(i,j)}  \geq \varepsilon \1 + (-\overline{\phi}_{(i,j)} - \varepsilon \1) \delta_{(i,j)}, %\label{eq:phi2}%
	\end{cases}
	\label{eq:phi1}%
\end{align}
which are equivalent to the logical constraint in \eqref{eq:phi_logic1};

{ \begin{align}
		[\phi_{(i,j)} \leq \overline{y}_i^j] \Leftrightarrow [\alpha_{(i,j)}^m = 1], \ \label{eq:ub_cons}%\\
		[\phi_{(i,j)} \geq \underline \phi_{(i,j)}^m] \Leftrightarrow [\beta_{(i,j)}^m = 1], \ %\label{eq:lb_cons}\\
		\gamma_{(i,j)}^m= \alpha_{(i,j)}^m \beta_{(i,j)}^m, %\label{eq:prod_cons}
\end{align}}
for $j=1,\dots,p_i$. Then, based on \cite[Eq. (4e) and (5a)]{bemporad99}, the constraints in \eqref{eq:ub_cons} %-\eqref{eq:prod_cons} 
hold if and only if
\fi 
\begin{align}
\hspace{-5pt}	\begin{cases}
		\phi_{(i,j)} - \overline \phi_{(i,j)}^m \leq (\overline \phi_{(i,j)}-\overline \phi_{(i,j)}^m)(1-\alpha_{(i,j)}^m), \\
		\phi_{(i,j)} - \overline \phi_{(i,j)}^m \geq  (-\overline \phi_{(i,j)} -\overline \phi_{(i,j)}^m ) \alpha_{(i,j)}^m, \\%\label{eq:ub_cons2}\\
		- \phi_{(i,j)} + \underline \phi_{(i,j)}^m \leq (\overline \phi_{(i,j)}+ \underline \phi_{(i,j)}^m)(1-\beta_{(i,j)}^m), \\%\label{eq:lb_cons1}\\
		-\phi_{(i,j)} + \underline \phi_{(i,j)}^m \geq  (-\overline \phi_{(i,j)}+ \underline \phi_{(i,j)}^m ) \beta_{(i,j)}^m, \\%\label{eq:lb_cons2}\\
		-\alpha_{(i,j)}^m + \gamma_{(i,j)}^m\leq 0, \ \ % \label{eq:prod_cons1}%\\
		-\beta_{(i,j)}^m + \gamma_{(i,j)}^m\leq 0,  %\label{eq:prod_cons2}
		\\
		\alpha_{(i,j)}^m + \beta_{(i,j)}^m - \gamma_{(i,j)}^m\leq 1, %\label{eq:prod_cons3}%
	\end{cases}
	\label{eq:ub_cons1}
\end{align}
for $m=1,\dots,r$, which are equivalent to the logical constraints \eqref{eq:phi_logic} \cite[Eq. (4e) and (5a)]{bemporad99}, with $\alpha_{(i,j)}^m, \beta_{(i,j)}^m \in \{0,1\}$, for $m=1,\dots, r,$ being additional binary variables;
\begin{align}
	\small 
	\hspace{-8.8pt}\begin{cases}
		\nu_{(i,j)}^m \geq -\overline \phi_{(i,j)} \gamma_{(i,j)}^m, & \nu_{(i,j)}^m \leq \phi_{(i,j)} + \overline \phi_{(i,j)}(1-\gamma_{(i,j)}^m), %\label{eq:z2}
		\\
		%\label{eq:z3}
		\nu_{(i,j)}^m \leq \overline \phi_{(i,j)} \gamma_{(i,j)}^m, &
		\nu_{(i,j)}^m \geq \phi_{(i,j)} - \overline \phi_{(i,j)}(1-\gamma_{(i,j)}^m), %\label{eq:z2}
	\end{cases}
	\label{eq:z1}
\end{align}
for all $m=1,\dots, r$, which equivalently represent $\nu_{(i,j)}^m=\gamma_{(i,j)}^m{\phi}_{(i,j)}$  \cite[Eq. (5.b)]{bemporad99}, and;
\begin{align}
	\begin{cases}
		\nu_{(i,j)}^{\psi} \geq \underline \psi_i \delta_{(i,j)}, \ \ \nu_{(i,j)}^{\psi} \leq \psi_i - \underline \psi_i(1-\delta_{(i,j)}), %\label{eq:z2}
		\\
		%\label{eq:z3}
		\nu_{(i,j)}^{\psi} \leq \overline \psi_i \delta_{(i,j)}, \ \
		\nu_{(i,j)}^{\psi} \geq \psi_i - \overline \psi_i(1-\delta_{(i,j)}), %\label{eq:z4}
	\end{cases}
	\label{eq:z3}
	%	\label{eq:z5}
	%	y_{(i,j)}^{\psi_j} \geq \underline \psi_j \delta_{(i,j)}, \ \ y_{(i,j)}^{\psi_j} \leq \psi_j - \underline \psi_j(1-\delta_{(i,j)}), %\label{eq:z2}
	%	\\
	%\label{eq:z3}
	%	y_{(i,j)}^{\psi_j} \leq \overline \psi_j \delta_{(i,j)}, \ \
	%	y_{(i,j)}^{\psi_j} \geq \psi_j - \overline \psi_j(1-\delta_{(i,j)}), \label{eq:z6}
\end{align}
which are equivalent to the equality $\nu_{(i,j)}^{\psi_i}=\delta_{(i,j)} \psi_i$.
\iffalse 
Finally, we collect all the auxiliary variables $y_i = \col(\{ \nu_{(i,j)}^{\psi_i}, \ \{\nu_{(i,j)}^m\}_{m=1}^r \}_{j\in \mc N_i^{\mathrm{g}}}) \in \bb R^{(1+r)|\mc N_i^{\mathrm{g}}|}$ and $\delta_i = \col( \{\delta_{(i,j)}, \{\alpha_{(i,j)}^m, \ \beta_{(i,j)}^m, \ \gamma_{(i,j)}^m\}_{m=1}^r \}_{j\in \mc N_i^{\mathrm{g}}}) \in \{0,1\}^{(1+3r)|\mc N_i^{\mathrm{g}}|}$, and we can compactly write \eqref{eq:ub_cons1}--\eqref{eq:z4} as follows:
\begin{align}
	G_i^u u_i + G_i^z y_i + G_i^\delta \delta_i \leq g_i, \forall i \in \mc N^{\mathrm{g}} \label{eq:pwa_loc_cons}
\end{align}
for some matrices $G_i^u \in \bb R^{(4+11r)|\mc N_i^{\mathrm{g}}| \times (2+|\mc N_i^{\mathrm{g}}|)}$, $G_i^z \in \bb R^{(4+11r)|\mc N_i^{\mathrm{g}}| \times (1+r)|\mc N_i^{\mathrm{g}}|}$, $G_i^\delta\in \bb R^{(4+11r)|\mc N_i^{\mathrm{g}}| \times (1+3r)|\mc N_i^{\mathrm{g}}|}$ and vector $g_i \in \bb R^{(4+11r)|\mc N_i^{\mathrm{g}}|}$. Furthermore,  \eqref{eq:gf_eq5}, \eqref{eq:psi1}, and \eqref{eq:psi2} are coupling constraints between node $i$ and $j$. 
\fi 

 {Thus,} we can compactly write {\eqref{eq:recip_flow} and \eqref{eq:gf_eq5},} for all $j \in \mc N_i^{\mathrm g}$, as  in \eqref{eq:hf_cons2}; \eqref{eq:deltam}, for all $j \in \mc N_i^{\mathrm g}$, as  in \eqref{eq:hf_cons1}; \eqref{eq:psi1}, for all $j \in \mc N_i^{\mathrm g}$, as  in \eqref{eq:gf_cons2}; and {\eqref{eq:phi_boundMIS}, \eqref{eq:ub_cons1}--\eqref{eq:z3},} for all $j \in \mc N_i^{\mathrm g}$, as  in \eqref{eq:gf_cons1}.%  {where $h_{i}^{\mathrm{cpl}}$, $h_{i}^{\mathrm{loc}}$, $g_{i}^{\mathrm{cpl}} $, and $g_{i}^{\mathrm{loc}}$ are affine functions.}

\section{Proofs}
\subsection{Proof of Lemma \ref{le:potential_game}}
\label{app:pf:le:potential_game}
{Based on the definition of generalized potential games in \cite[Def. 2.1]{facchinei11decomposition}, we need to show that (a.) the global feasible set $\bm{\mc U}$ is non-empty and closed and (b.) the function $P$ in \eqref{eq:P} is a potential function of the game in \eqref{eq:ED_game1}. The set $\bm{\mc U}$, which is non-empty by Assumption \ref{as:nonempty_globalset}, is closed since it is constructed from the intersection of closed half spaces and hyperplanes as we have affine non-strict inequality and equality constraints.}  

Next, by the definition of $J_i$ in \eqref{eq:Ji}, %for each $i \in \mc I$, 
%\begin{multline}
%	\nabla_{x_i} J_i(\bm x) = \nabla_{x_i} f_i^{\mathrm{loc}}(x_i) + (\sum_{j \in \mc I}Q^\mathrm{ e} \Xi_j^{\mathrm e} x_j)\Xi_i^{\mathrm e}x_i \\
%	+(\sum_{j \in \mc I}Q^\mathrm{ g} \Xi_j^{\mathrm g} x_j)\Xi_i^{\mathrm e}x_i.
%\end{multline}
the pseudogradient mapping of the game in \eqref{eq:ED_game1} is 
\begin{align}
	\col((\nabla_{x_i}J_i(\bm x))) := \col((\nabla_{x_i} f_i^{\mathrm{loc}}(x_i))_{i \in \mc I}) + \bm D \bm x, \notag
\end{align}
where $\bm D$ is a symmetric matrix, i.e., its block component $(i,j)  {\in \mc I \times \mc I}$ is defined by
\begin{align}
	[\bm D ]_{i,j} = 
	\begin{cases}
		2 ((\Xi_i^{\mathrm e})^\top Q^{\mathrm e} \Xi_i^{\mathrm e} + (\Xi_i^{\mathrm g})^\top Q^{\mathrm g} \Xi_i^{\mathrm g}), & \text{if } i=j,\\
		(\Xi_i^{\mathrm e})^\top Q^{\mathrm e} \Xi_j^{\mathrm e}+ (\Xi_i^{\mathrm g})^\top Q^{\mathrm g} \Xi_j^{\mathrm g}, & \text{otherwise.}
	\end{cases} \notag
\end{align}
%Furthermore,  it must hold that $\nabla_{\bm x} P(\bm x) = \col((\nabla_{x_i} J_i(\bm x))_{i \in \mc I})$. 
%Hence $P(\bm x)$ defined in \eqref{eq:P} asserts that the game in \eqref{eq:ED_game1} is an exact potential one since 
{On the other hand, by the definition of $P$ in \eqref{eq:P}, it holds that}
%\iffalse 
\begin{align}
	P(\bm x) 	=& \tfrac{1}{2} \sum_{i\in\mathcal{I}}(J_i(\bm x) + f_i^{\mathrm{loc}}(x_i) +x_i^\top D_i x_i) )\notag \\
	=& \sum_{i\in \mc I} (f_i^{\mathrm{loc}}(x_i))  +\tfrac{1}{2}(x_i^\top D_i x_i + f_i^{\mathrm{cpl}}(\bm x))) \notag \\
	=&\sum_{i\in \mc I} f_i^{\mathrm{loc}}(x_i) + \tfrac{1}{2}\bm x^\top\bm D \bm x \notag.
\end{align}
%\fi
%Therefore, 
Thus, $\nabla_{\bm x} P(\bm x) = \col((\nabla_{x_i} J_i(\bm x))_{i \in \mc I})$, {implying that $P$ is an exact potential function of the game in \eqref{eq:ED_game1}.} %, satisfying the definition of an exact potential game \cite{bibid}.
\iffalse 
For each $i \in \mc I$, the cost function of each agent can be rewritten as
\begin{multline*}
	J_i(\bm x) =f_i(x_i) + x_i^\top ( (\Xi_i^{\mathrm e})^\top Q^{\mathrm e}  \Xi_i^{\mathrm e} + (\Xi_i^{\mathrm g})^\top Q^{\mathrm g}  \Xi_i^{\mathrm g})x_i \\
	+ \Big(\sum_{j \in \mc I \backslash \{i\}}( (\Xi_i^{\mathrm e})^\top Q^{\mathrm e}  \Xi_j^{\mathrm e} + (\Xi_i^{\mathrm g})^\top Q^{\mathrm g}  \Xi_j^{\mathrm g})x_j\Big)^\top x_i,
\end{multline*}
Hence, for each $i \in \mc I$ and any $(u_i, \bm u_{-i}), (\tilde{u}_i,\bm{u}_i) \in (\prod_{i\in\mc I} \mc L_i) \cap \mc C$,
$$P(x_i,\bm x_{-i}) - P(\tilde x_i,\bm x_{-i}) = J_i(x_i,\bm x_{-i}) - J_i(\tilde x_i,\bm x_{-i}).$$
\fi

\subsection{Proof of Lemma \ref{le:mon_pseudograd_P_conv}}
\label{app:pf:le:mon_pseudogradient}
\begin{enumerate}[(i)]
	\item We have that 
	\begin{align}
		\col((\nabla_{x_i}J_i(\bm x))) := \col((\nabla_{x_i} f_i^{\mathrm{loc}}(x_i))_{i \in \mc I}) + \bm D \bm x, \notag
	\end{align}
	{where $\bm D$ is as defined in the proof of Lemma \ref{le:potential_game}.}
	The operator $\col((\nabla_{x_i} f_i^{\mathrm{loc}}(x_i))_{i \in \mc I})$ is monotone since it is a concatenation of monotone operators $\nabla_{x_i} f_i^{\mathrm{loc}}(x_i))$, for all $i \in \mc I$, \cite[Prop. 20.23]{bauschke11} as $f_i^{\mathrm{loc}}$, for each $i \in \mc I$, is differentiable and convex by definition. Furthermore, $\bm D$ is positive semidefinite by construction. Therefore, both $	\col((\nabla_{x_i}J_i(\bm x))_{i \in \mc I}) $ and $F$, which concatenates $	\col(\nabla_{x_i}J_i(\bm x))_{i \in \mc I}) $ and $\0$, are monotone.
	\item {As in the proof of Lemma \ref{le:potential_game}, $\nabla_{\bm x} P(\bm x) = \col((\nabla_{x_i} J_i(\bm x))_{i \in \mc I})$, which is monotone by Lemma \ref{le:mon_pseudograd_P_conv}.(\ref{le:mon_pseudogradient}). Hence, $P$ is a convex function.}
\end{enumerate}

\subsection{Proof of Proposition \ref{prop:0opt_is_MIGNE2}}
\label{app:pf:prop:0opt_is_MIGNE}
\begin{enumerate}[(i)]
	\item 
{	Since $\hat {\bm u}$ is a feasible point of the convexified game in \eqref{eq:ED_game_c}, then $\hat {\bm u}$ satisfies all the constraints in the original game \eqref{eq:ED_game1} except possibly the integrality constraints $z_i \in \{0,1\}^{n_{z_i}}$, for all $i \in \mc I$. Therefore, by construction, $\bm u^\star=(\bm x^\star, \bm y^\star, \bm z^\star)$ satisfies all the constraints but the MISOC gas flow constraints \eqref{eq:gf_MIS2}--\eqref{eq:gf_MIS5}, equivalently \eqref{eq:gf_MIS6}. By the definition of the inequality constraints \eqref{eq:gf_MIS7}, a solution to Problem \eqref{eq:qp_pressure2} satisfies \eqref{eq:gf_MIS6} if and only if $\tilde{\bm \tau} = \0$.}
	\item From point (i), $\bm u^\star$ is a feasible point of the original game if and only if $\tilde{\bm \tau}=\0$. Furthermore, we observe that the cost functions in the original game \eqref{eq:ED_game1} and those in the convexified game \eqref{eq:ED_game_c} are equal and only depend on $\bm x$.  {By considering that $\hat{\bm u}$ is a variational GNE of the game in \eqref{eq:ED_game_c}, implying that it is also a solution to Problem \eqref{eq:opt_P_cvx}, and  {that} the optimal value of Problem \eqref{eq:opt_P_cvx} is a lower bound of Problem \eqref{eq:opt_P}, we can conclude that $\bm u^\star$ is a solution to Problem \eqref{eq:opt_P} as $P(\bm u^\star) = P(\hat{\bm u})$. Hence, $\bm u^\star$ is an exact} GNE of the original game, i.e., {the inequality in \eqref{eq:eps_GNE} holds} with $\varepsilon=0$  {\cite[Thm 2]{sagratella17}.} 
	\item When $J_\psi(\tilde{\bm \psi})=0$, which also consequently implies that $\| \tilde \tau\|_{\infty} = 0$, the SOC gas flow constraint $ \nu_{(i,j)} \geq	\phi_{(i,j)}^2/c_{(i,j)}^{\mathrm f}$, for each $(i,j) \in \mc E^{\mathrm g}$, is tight, i.e., satisfied with an equality.
\end{enumerate}

\subsection{Proof of Lemma \ref{le:tree_G}}
\label{app:pf:le:tree_G}
	A set of solutions to {\eqref{eq:lin_eq_psi_h2b}} exists if and only if $\operatorname{rank}\left(\begin{bmatrix}
	E(\tilde{\bm{\delta}}_h) & \bm \theta_h(\hat{\bm{\phi}_h})
\end{bmatrix}\right) = \operatorname{rank}(E(\tilde{\bm{\delta}}_h))$, for all $h \in \mc H$. Since we assume $\mc G^{\mathrm g}$ is connected, $\mc G^{\mathrm g}$ is either a minimum tree or  {is not a minimum tree.} When $\mc G^{\mathrm g}$ is a minimum tree, {$|\mc E^{\mathrm g}| = 2(N-1)$, as we label an edge between node $i$ and $j$ twice, i.e. $(i,j)$ and $(j,i)$. Therefore,} $\operatorname{rank}(E(\tilde{\bm{\delta}}_h)) = N-1$. Moreover, $\begin{bmatrix}
	E(\tilde{\bm{\delta}}_h) & \bm \theta_h(\hat{\bm{\phi}_h})
\end{bmatrix} \in \bb R^{|\mc E^{\mathrm g}| \times (N+1)}$ {and $\phi_{(i,j),h}^2 = \phi_{(j,i),h}^2$, for all $(i,j) \in \mc E^{\mathrm g}$.} Thus, $\operatorname{rank}\left(
\begin{bmatrix}
	E(\tilde{\bm{\delta}}_h) & \bm \theta_h(\hat{\bm{\phi}_h})
\end{bmatrix}
\right) = N-1$.

\subsection{Proof of Theorem \ref{thm:character_sol}}
\label{app:pf:thm:character_sol}
%\com{complete the proof}
%\begin{enumerate} 
%\item Notation: jointly convex feasible set.
%\item %\com{Show that for $\ell =1$, the solution is a lower bound of the minimization of $P(\bm x)$ over the jointly convex feasible set, which is a relaxed. }

% Clarify: v-GNE = solution to the minimization of $P(\bm x)$. 
By Lemma \ref{le:mon_pseudograd_P_conv}.(\ref{le:P_convex}), $P(\bm x)$ in \eqref{eq:P} is a convex function. {Therefore,} a variational GNE of the convexified game in \eqref{eq:ED_game_c_pen} with $\rho = \rho^{(1)}=0$ (or equivalently the game in \eqref{eq:ED_game_c}), denoted by $(\hat{\bm x}^{(\ell)}, \hat{\bm y}^{(\ell)}, \hat{\bm z}^{(\ell)})$, is a solution to Problem {\eqref{eq:opt_P_cvx}}, 
%\begin{equation}
%	\begin{aligned}
%		\begin{cases}
%			\min_{\bm u} \quad & P(\bm x) \\
%			\operatorname{s.t.} \quad & \bm u \in \prod_{i\in\mc I} \mc L_i \cap \bm{\mc C},
%		\end{cases}
%	\end{aligned}
%	\label{eq:P_convex}
%\end{equation}
%since the Karush-Kuhn-Tucker optimality conditions of \eqref{eq:opt_P_cvx} and those of a v-GNE of the game in \eqref{eq:ED_game_c} coincide. 
{which} is a convex relaxation of {Problem \eqref{eq:opt_P}.}
%\begin{equation}
%	\begin{aligned}
%		\begin{cases}
%			\min_{\bm u} \quad & P(\bm x) \\
%			\operatorname{s.t.} \quad & \bm u \in \bm{\mc U}.
%		\end{cases}
%	\end{aligned}
%	\label{eq:P_MI}
%\end{equation} 
Therefore, {by denoting with $\bm u^\ast =(\bm x^\ast, \bm y^\ast, \bm z^\ast)$ a solution to Problem \eqref{eq:opt_P}, which is {an exact} GNE of the game in \eqref{eq:ED_game1}, we have that} 
\begin{equation}
P(\tilde{\bm x}^{(1)}) =	P(\hat{\bm x}^{(1)}) \leq P(\bm x^\ast), \label{eq:ineq_main_theorem1}
\end{equation}
where {the equality holds since} $\tilde{\bm x}^{(\ell)}=\hat{\bm x}^{(\ell)}$ (see Step 4 of Algorithm \ref{alg:2stage_meth_iterate}). %Hence, we conclude that $\tilde{\bm u}^{(1)}$ provides a lower bound of Problem \eqref{eq:P_MI}.
%	\item \com{Show that for $\overline{\ell}$, the solution is an upper bound of the minimization of $P(\bm x)$.} 
Next, we observe {from} Proposition \ref{prop:0opt_is_MIGNE2}.(i) {that} $\tilde{\bm u}^{(\overline \ell)}$ is a feasible point but {is} not necessarily a solution to Problem \eqref{eq:opt_P_cvx} (nor a GNE) since the cost functions considered in Step 1 is $\tilde{J}_i$, for all $i \in \mc I$. {Thus,} it holds that 
\begin{equation}
	P(\bm x^\ast) \leq P(\tilde{\bm x}^{(\overline \ell)}). \label{eq:ineq_main_theorem2}
\end{equation}

%	\item \com{Use the fact that $P$ is an exact potential function and the definition of $\varepsilon$-GNE to obtain the bound in \eqref{eq:eps_bound}. }

Since $P$ is an exact potential function, it holds that, for each $i \in \mc I$ and any $(x_i,y_i,z_i)\in \mc L_i \cap \mc C_i(\tilde{\bm u}_{-i}^{(\overline \ell)})\cap \bb R^{n_{x_i}+n_{y_i}}\times \{0,1\}^{n_{z_i}}$, we have that
\begin{align*}
	J_i(\tilde{\bm x}^{(\overline \ell)}) - J_i(x_i,\tilde{\bm x}_{-i}^{(\overline \ell)}) &= P(\tilde{\bm x}^{(\overline \ell)}) - P(x_i,\tilde{\bm x}_{-i}^{(\overline \ell)}) \\
	&\leq P(\tilde{\bm x}^{(\overline \ell)}) - P(\bm x^\ast) \\
	&\leq P(\tilde{\bm x}^{(\overline \ell)}) - P(\tilde{\bm x}^{(1)}) = \varepsilon,
\end{align*}
where the first inequality {holds} since $\bm u^\ast$ is an optimizer of Problem \eqref{eq:opt_P}, and the second inequality is obtained by combining \eqref{eq:ineq_main_theorem1} and \eqref{eq:ineq_main_theorem2}.
%	\end{enumerate}

\iffalse 

\subsection{Proof of Proposition \ref{prop:0opt_is_MIGNE}}
%\begin{proof}
	The proofs are analogous to those of Proposition \ref{prop:0opt_is_MIGNE2}(i)--(ii), where ${\bm{u}}^\star$ satisfies all except possibly the gas flow equality constraints in \eqref{eq:gf_eq6}. By the definition of $\tilde J_\psi$ in \eqref{eq:qp_pressure_cost}, $(\bm y^\star,\bm z^\star)$ satisfies \eqref{eq:gf_eq6} if and only if $\tilde J_\psi(\bm \psi^\star)=0$.
%\end{proof}
%\newpage
\color{black}
\fi 

\subsection{Proof of Proposition \ref{prop:sn_0sol}}
\label{app:pf:prop:sn_0sol}
%	For clarity, let us breakdown \eqref{eq:lin_eq_psi} into $h$ systems of linear equations, i.e.,
%	\begin{equation}
	%		 E(\tilde{\bm{\delta}}_h) \bm \psi_h -\bm \theta_h(\hat{\bm{\phi}_h}, %\tilde{\bm{\delta}}^{\mathrm{pwa}})=0, 
	%		\label{eq:lin_eq_psi_h}
	%	\end{equation}
%for all $h \in \mc H$. Therefore, $\bm \psi^0$ is composed of the subvectors $\bm \psi_h^0$, for all $h \in \mc H$, each of which is a particular solution to \eqref{eq:lin_eq_psi_h}. 
By the definition of the matrix $ E(\tilde{\bm{\delta}}_h)$ and Assumption \ref{as:tree_G}, its null space is $\{ \1_N \}$. Hence, for each $h \in \mc H$, the solutions to \eqref{eq:lin_eq_psi_h2} can be described by $\bm \psi_h^0 + \xi \1_N,$ for any $\xi \in \bb R$.  {Therefore, Problem \eqref{eq:qp_pressure} has at least a solution if and only if there exists $\xi \in \bb R$ such that }
%$$
\begin{equation}
	\col(\{\underline{\psi}_i\}_{i \in \mc I}) \leq \bm \psi_h^0 + \xi \1_N \leq \col(\{\overline{\psi}_i\}_{i \in \mc I}),
	\label{eq:ineq_particular_sol}
\end{equation}
%$$
since for any $\xi \in \bb R$, $J_\psi(\bm \psi_h^0 + \xi \1_N)=0$. For each $h \in \mc H$, let us now consider another particular solution, 
\begin{equation}
	\bar{\bm \psi}_h^0 = \bm \psi_h^0 - \1_N \min_{i \in \mc I} [\bm \psi_h^0]_i,
	\label{eq:bar_psi_0}
\end{equation}
i.e., $\xi =  -\min_{i \in \mc I} [\bm \psi_h^0]_i$. Hence, $\min_{i\in \mc I} [\bar{\bm \psi}_h^0]_i = 0$ and $\max_{i\in \mc I} [\bar{\bm \psi}_h^0]_i = \max_{i\in \mc I} [{\bm \psi}_h^0]_i-\min_{i \in \mc I} [\bm \psi_h^0]_i \geq 0$. Now, let us consider the indices $j$ and $k$ as defined in Proposition \ref{prop:sn_0sol}, i.e., $j \in \argmin_{\ell \in \operatorname{argmax}_{i \in \mc I} [\bm \psi_h^0]_i} \overline{\psi}_\ell$, and $k \in \operatorname{argmax}_{\ell \in \operatorname{argmin}_{i \in \mc I} [\bm \psi_h^0]_i}\overline{\psi}_\ell$ and let us substitute $\bm \psi_h^0$ with $\bar{\bm \psi}_h^0$ in \eqref{eq:ineq_particular_sol}.  Since $\underline{\psi}_i \geq 0$, for any $i \in \mc I$, and $\min_{i\in \mc I} [\bar{\bm \psi}_h^0]_i = 0$, the first inequality in \eqref{eq:ineq_particular_sol} is satisfied if and only if $\xi \geq \underline{\psi}_k \geq 0$. Furthermore, the second inequality is satisfied if and only if $\xi \leq \overline{\psi}_j - [\bar{\bm \psi}_h^0]_j$. Hence, there exists $\xi \geq 0$ if and only if
\begin{align}
	\underline{\psi}_k \leq \overline{\psi}_j - [\bar{\bm \psi}_h^0]_j &\Leftrightarrow  [\bar{\bm \psi}_h^0]_j  \leq \overline{\psi}_j -\underline{\psi}_k \notag \\
	&\Leftrightarrow [\bm \psi_h^0]_j -  [\bm \psi_h^0]_k \leq \overline{\psi}_{j} - \underline{\psi}_{k}, \notag 
\end{align}
where the last implication follows the definition of $\bar{\bm \psi}_h^0$ in \eqref{eq:bar_psi_0}.

\bibliographystyle{IEEEtran}
%\bibliography{IEEEfull,biblio}
\bibliography{ref}

\end{document}